\documentclass[a4wide,10pt]{article}
\usepackage[utf8]{inputenc}
\usepackage{amsmath,amsfonts,amssymb,amsthm}
\usepackage{array}
\usepackage{hyperref}
\usepackage{bbm}
\usepackage{orcidlink}
\usepackage{float}

\newtheorem{thm}{Theorem}[section]
\newtheorem{prop}[thm]{Proposition} 
\newtheorem{ass}[thm]{Assumption}
\newtheorem{remark}{Remark}
\newcommand{\R}{\mathbb R}

\title{Sparse control in microscopic and mean-field leader-follower models}

\date{}

\author{
Melanie Harms\footnote{Lehrstuhl für Algebra und Zahlentheorie, RWTH Aachen University, Aachen, Germany; \texttt{melanie.harms@rwth-aachen.de}, \texttt{eva.zerz@math.rwth-aachen.de}}, 
$\;$
Michael Herty\footnote{Institut für Geometrie und Praktische Mathematik, RWTH Aachen University, Aachen, Germany; \texttt{herty@igpm.rwth-aachen.de}}, 
$\;$
Chiara Segala\footnote{Faculty of Informatics, Università della Svizzera italiana -- USI, Lugano, Switzerland; \texttt{chiara.segala@usi.ch}}, 
$\;$
Eva Zerz$^*$
}

\begin{document}
\maketitle
\begin{abstract}
This work investigates the decay properties of Lyapunov functions in leader-follower systems seen as a sparse control framework. Starting with a microscopic representation, we establish conditions under which the total Lyapunov function, encompassing both leaders and followers, exhibits exponential decay. The analysis is extended to a hybrid setting combining a mean-field description for followers and a microscopic model for leaders. We identify sufficient conditions on control gain and interaction strengths that guarantee stabilization of the linear system towards a target state. The results highlight the influence of sparse control and interaction parameters in achieving coordinated behavior in multi-agent systems.
\end{abstract}



\section{Introduction}\label{sec:intro}

The study of particle systems subject to control has become an active field of research. For comprehensive discussions, we refer to \cite{MR3642940, MR3969953, MR3308728}, which include recent books on the topic, along with the references within them. The theory of open and closed-loop control applied to large-scale (microscopic) systems, as well as their mean-field limits, has been extensively explored. Key references in this area include \cite{MR4030515, MR4642632, MR3679340, MR4226993, MR4028474}. 

In addition, various approaches have been proposed to develop new control strategies, particularly in the context of interacting particle systems and the effects of sparsity in control. For instance, instantaneous control applied over a single time horizon leads to explicit feedback laws \cite{MR4046175, MR3351435}. These strategies are sparse in the sense that not the entire time horizon is incorporated into the control. A decentralized control approach, as discussed in \cite{MR4028780}, introduces sparsity by limiting interactions among particles. 

In the so-called moment-driven control, sparsity is introduced by focusing on a reduced system that governs, for example, the mean or variance of the ensemble, rather than solving the full system \cite{MR4399019}. Sparse control strategies are also closely linked to dynamics based on leader-follower interactions. Typically, such interactions are approximated using mean-field models, where the leader's influence is treated as a control strategy aimed at steering the system towards a desired state. Several examples of this approach have been discussed in the literature, e.g., \cite{MR4797665, MR4409801, MR3945372, AlbiFerrareseSegala21, MR2552289, MR3542027, MR3195343, MR3268062}, with mean-field limits rigorously studied in \cite{MR3268059}.

It is important to note, however, that in most cases, the interaction is prescribed. In contrast, the present work derives a feedback control strategy for the leaders based on a Lyapunov function approach. Additionally, a related study on sparse control \cite{MR3195343} employs $\ell_1$-regularization to guide a self-organizing system. In that work, it was observed that a few agents are sufficient to significantly influence the overall dynamics. These findings are also demonstrated in the current work, albeit through a different approach to deriving the feedback control.

In particular, in this work we consider first a microscopic model where each agent interacts with others based on a symmetric interaction kernel and is subject to an external control input. The goal is to design control mechanisms that ensure convergence to a consensus state while minimizing the number of controlled agents. 
In order to do that, we investigate the stabilization properties of a leader-follower systems seen as sparse control strategies.
Additionally, to address large-scale systems, we extend our analysis to a hybrid setting, where leaders are modelled microscopically, and followers are described via a mean-field formulation. We establish conditions under which the Lyapunov function of the linear system exhibits exponential decay, ensuring stabilization. We then validate our theoretical results through numerical simulations, demonstrating the effectiveness of the proposed control strategies.

The paper is organized as follows. In Section~\ref{sec:micro_nonlin}, we introduce the microscopic model and discuss its control properties. In particular, we present the linearized system and we derive conditions for stabilization as well as appropriate linear state feedback using spectral techniques. We analyze also the correspondent mean-field model. Section~\ref{sec:leader_follower} focuses on the leader-follower system. Finally, Section~\ref{sec:lyap} presents a Lyapunov stability analysis, followed by the numerical tests in Section~\ref{sec:numtests}, and conclusions are drawn in Section~\ref{sec:conclusion}.

\section{Multi-agent system with control action}\label{sec:micro_nonlin}

We consider the control of high-dimensional non-linear systems of interacting particles, which can model the dynamics of multi-agent systems exhibiting self-organization. In particular, the state of each agent $x_i:=x_i(t)\in\mathbb{R}$ evolves through interactions with other agents and an external control signal $u:=u(t)\in\mathbb{R}$. The system dynamics are governed by the equation
\begin{equation}\label{nonlin_dynamics}
	\begin{split}
	&\dot{x}_i = \frac{1}{N}\sum_{j=1}^N P(x_i,x_j)(x_j - x_i) + b_i u,
	\\
	&x_i(0)=x_{i}^0,\qquad
	i=1,\ldots,N,
		\end{split}
\end{equation}
where $P:\R\times\R\to\R$ is a symmetric interaction kernel governing the influence of agent $j$ on agent $i$, and 
$b_i\in\R$ is a parameter through which the external control $u$ acts specifically on agent $i$. 
The system describes the evolution of $N$ agents undergoing binary exchanges of information, weighted by the interaction kernel.

In the context of this work, we are particularly interested in driving the system to a \textit{consensus equilibrium}, where all agents share the same state. A consensus state $x_c \in \mathbb{R}^N$ is defined as
\[
x_c = c \,\mathbbm{1}_{N\times 1},
\]
where $c \in \R$ is a constant and $\mathbbm{1}_{N\times 1}$ is the all-ones vector of size $N\times 1$, meaning that all agents converge to the same value: $x_1 = x_2 = \ldots = x_N = c$. This equilibrium is the desired outcome of the control design.

\begin{ass}\label{ass:symmetry}
	To model the interaction between agents, we assume the kernel $P(x, y)$ to be symmetric, i.e.,
	\[
	P(x, y) = P(y, x) \quad \text{for all} \ x, y \in \R.
	\]
	This symmetry reflects the idea that the influence between two agents is mutual.
\end{ass}
\begin{remark}
    Under Assumption \ref{ass:symmetry}, if the system is uncontrolled (i.e., $u(t) \equiv 0$), the dynamics conserves the mean state over time.
    In particular, if  the system naturally drive the agents toward consensus in the absence of control, then the consensus point is the mean of the initial conditions:
    \[
    c = \frac{1}{N} \sum_{i=1}^N x_i^0,
    \]
    i.e. the consensus is determined by the initial distribution of agent states. However, our objective is to introduce an external control $u(t)$ to drive the system to a specific consensus, potentially different from the natural outcome. 
\end{remark}
The control design we consider focuses on \textit{sparse control}, where the goal is to affect the system using as few control inputs as possible, thus minimizing the control effort while still achieving consensus. This is particularly important in large systems, where it may be impractical or costly to control every agent. 

The main challenges in addressing this control problem arise from the high dimensionality of the system and the non-linearity introduced by the interaction kernel $P(x_i, x_j)$. As the number of agents $N$ increases, direct numerical solutions become computationally expensive. Moreover, designing sparse controls that are both effective and scalable adds further complexity to the problem.

We start proposing an approach for solving this sparse control problem 
using eigenvalue analysis. By analyzing the spectral properties of the system, we identify control strategies that drive the system to consensus. Furthermore, we employ mean-field approximations to reduce the complexity of the system when $N$ becomes large, providing a tractable solution for large-scale multi-agent systems.

\subsection{Eigenvalue analysis of the linearized system}\label{sec:linearization}

To facilitate the synthesis of a feedback control law for the non-linear dynamics introduced in \eqref{nonlin_dynamics}, we begin by linearizing the system around an equilibrium point, which corresponds to a consensus state of the agents.

We define the vector-valued function $f(x): \R^{N} \rightarrow \R^{N}$, which describes the interaction dynamics of the agents, as follows:
\begin{equation*}
	f_i(x) = \frac{1}{N}\sum_{j=1}^N P(x_i, x_j)(x_j - x_i),\qquad i = 1, \ldots, N.
\end{equation*}
The function $f_i(x)$ captures the influence of all other agents on the state of agent $i$.

We linearize the dynamics around a consensus point $x_c$. 
At this equilibrium, all agents share the same state and the non-linear interaction term satisfies $f(x_c) = 0_{N\times 1}$.

The communication kernel $P(x_i, x_j)$ evaluated at the consensus point is a constant value $\bar{p}\in\R$, i.e.,
\[
P(c, c) = \bar{p},
\]
that represents the influence between agents at equilibrium. 

The first-order Taylor expansion of $f(x)$ around the equilibrium point $x_c$ yields the linearized interaction term:
\[
f(x) \approx A (x - x_c) = A x,
\]
where $A = \frac{\partial f}{\partial x}\Big|_{x = x_c} \in \R^{N \times N}$ is the Laplacian matrix of the system, which captures the interaction structure of the agents in the linear regime. The matrix $A$ has the following form:
\begin{align} \label{AB}
	A = -\bar p I_N + \frac{\bar p}{N} \mathbbm{1}_{N \times N}, \qquad
	A_{ij} =
	\begin{cases} 
		&\frac{\bar{p}(1-N)}{N}, \qquad i = j, \\
		&\frac{\bar{p}}{N}, \qquad\qquad i \neq j,
	\end{cases}
\end{align}
where $I_N$ is the identity matrix of size $N \times N$, and $\mathbbm{1}_{N \times N}$ is the all-ones matrix, reflecting the uniform influence between all agents at the consensus state. This structure of $A$ ensures that $A(x - x_c) = A x$ since $x_c$ represents a consensus point, where all agents share the same value.

Therefore, the linearized system around the equilibrium point is governed by the following equation:
\begin{equation}\label{lin_dynamics}
	\dot{x} = A x, \qquad \dot{x}_i = \frac{1}{N}\sum_{j=1}^N \bar{p}(x_j - x_i).
\end{equation}
In this linearized regime, the agents evolve under the influence of a uniform interaction strength $\bar{p}$, and the dynamics are described by the Laplacian matrix $A$, which is commonly used to model consensus problems in multi-agent systems.

To classify the equilibrium point \(x_c\) of the linearized system \eqref{lin_dynamics}, we analyze the eigenvalues of the matrix \(A\). The eigenvalue structure is crucial for understanding the stability of the system under control.

\begin{ass}
	We assume that the bounded value \(P(c,c) = \bar{p}\) is strictly positive.
\end{ass}
The matrix \(A\) possesses the following eigenvalue characteristics:
\begin{itemize}
	\item One eigenvalue \(\lambda_1 = 0\) with the corresponding eigenvector \(v_1 = \mathbbm{1}_{N\times 1}\),
	\item \(N-1\) eigenvalues \(\lambda_i = -\bar{p}\) with eigenvectors \(v_i = e_1 - e_i\), for \(i = 2, \ldots, N\).
\end{itemize}
Since \(\lambda_1\) is not strictly negative, we cannot guarantee that the equilibrium point \(x_c\) is asymptotically stable. To achieve asymptotic stability, our objective is to shift \(\lambda_1\) to a negative value. To do this, we consider the controlled linearized system of \eqref{nonlin_dynamics} given by:
\begin{equation}\label{eq:lin_contr_sys}
	\dot{x} = A x + B u,
\end{equation}
where \(A \in \mathbb{R}^{N \times N}\) is defined in \eqref{AB} and \(B=(b_1,\ldots,b_N)^T \in \mathbb{R}^N\).

This system is not fully controllable, but it is stabilizable if and only if the matrix $(A,B)\in\R^{(N+1)\times N}$ has rank $N$. The latter is equivalent to $\sum_{j = 1}^{N} b_j \neq 0$, which motivates the next assumption.
\begin{ass}\label{sum:bi}
	We assume that the sum of the components of \(B\) is non-zero:
	\[
	\sum_{j = 1}^{N} b_j \neq 0.
	\]
\end{ass}
To stabilize the system \eqref{eq:lin_contr_sys}, we employ a linear state feedback control law, which only depends on the mean state $\frac{1}{N}\sum_{i=1}^Nx_i$ and conserves the equilibrium $x_c$ of the uncontrolled system. All possible feedback of this type have the following form
\begin{equation}\label{eq:u_feedback}
	u = F (x - x_c) = \frac{k}{N} \sum_{j = 1}^{N} (x_j - c) = k \left( \frac{1}{N} \sum_{j = 1}^{N} x_j - c \right),
\end{equation}
where
\[
F = \frac{k}{N} \mathbbm{1}_{1 \times N} \in \mathbb{R}^{1 \times N}
\]
and $k\in\R$ is arbitrary.
Substituting this control law into the system, we have:
\begin{equation*}
	\dot{x} = A x + B F (x - x_c) = A x_c + (A + B F) (x - x_c).
\end{equation*}
Next, we seek the eigenvalues of the matrix \(A + BF\), which can be expressed as:
\begin{align*}
	A + BF = -\bar{p} I_N + \frac{\bar{p}}{N} \mathbbm{1}_{N \times N} + B\frac{k}{N} \mathbbm{1}_{1 \times N} ,
	\\[1ex]
	(A + BF)_{ij} =
	\begin{cases} 
		\frac{\bar{p}(1-N)}{N} + \frac{k}{N} b_i, & \quad i = j, \\
		\frac{\bar{p}}{N} + \frac{k}{N} b_i, & \quad i \neq j.
	\end{cases}
\end{align*}
The modified system \(A + BF\) now has the following eigenvalue structure, which can be checked by a simple computation:
\begin{itemize}
	\item One eigenvalue \(\lambda_1 = \frac{k}{N} \sum_{j = 1}^{N} b_j\) with the eigenvector \(v_1 = \mathbbm{1}_{N\times 1} + \frac{k}{\bar{p}} B\), 
	\item \(N-1\) eigenvalues \(\lambda_i = -\bar{p}\) with eigenvectors \(v_i = e_1 - e_i\) for \(i = 2, \ldots, N\).
\end{itemize}

Assumption \ref{sum:bi} allows us to appropriately choose the parameter \(k\) (either positive or negative) based on the sign of \(\sum_{j = 1}^{N} b_j\neq0\) such that \(\lambda_1<0\). Thus, we can ensure that the controlled linear dynamics become asymptotically stable. Consequently, the controlled nonlinear dynamics, when incorporating the linear feedback law \eqref{eq:u_feedback}, i.e.,
\begin{equation}\label{eq:dyn_contr}
    \dot{x}_i = \frac{1}{N}\sum_{j=1}^N P(x_i,x_j)(x_j - x_i) + b_i\frac{k}{N} \sum_{j=1}^N(x_j-c),
\end{equation}
will also exhibit asymptotic stability around the consensus state $x_c$ (see for instance \cite{MR1640001}).

We are particularly interested in investigating sparsity and mean-field approximations in the context of control strategies for multi-agent systems. In light of this, we examine the following two hypothetical cases for the matrix \(B\):
\begin{align*}
	B &= e_j \quad \rightarrow \quad \lambda = -\bar{p}; \ \frac{k}{N}, \qquad (Bu)_i = k \left( \frac{1}{N} \sum_{l = 1}^{N} x_l - c \right) \ \text{for} \ i=j.
	\\
	\\
	B &= \mathbbm{1}_{N\times 1} \quad \rightarrow \quad \lambda = -\bar{p}; \ k, \qquad (Bu)_i = k \left( \frac{1}{N} \sum_{l = 1}^{N} x_l - c \right) \ \text{for all} \ i.
\end{align*}
In the first case, where \(B = e_j\), we encounter sparsity since control is applied to only one agent \(j\) at a time. However, when considering the mean-field limit as \(N\) approaches infinity, the eigenvalue \(\frac{k}{N}\) tends to zero, making it difficult to draw conclusions about stability in this context. 

In the second case, where \(B = \mathbbm{1}_{N \times 1}\), we find a well-defined mean-field limit as the average of the particles transitions to an integral, allowing us to interpret it in terms of mean densities. Here, the eigenvalues remain manageable, and we can adjust \(k\) to ensure that they are all negative, which facilitates stability. However, this case lacks sparsity since control is applied uniformly to all agents, which is impractical for large systems.

Given these insights, our goal is to identify an intermediate approach that achieves both effective control and sparsity, allowing us to control only a percentage of the agents while maintaining the desired stability properties, we investigate this approach in Section \ref{sec:leader_follower}.

In summary, through the eigenvalue analysis of the controlled linear system, we have established the conditions under which stability can be achieved, setting the stage for effective control strategies.

\subsection{Mean-field limit}\label{sec:mean_field}

In order to assess the validity of our results in the scenario where a large number of agents is present (i.e., \(N \gg 1\)), we explore our approach in the context of the {\em mean-field limit}. The mean-field theory provides a powerful framework for describing the collective behavior of a large ensemble of interacting particles by approximating the dynamics of the system with a continuous distribution rather than discrete individuals. This approach becomes particularly useful when the number of agents increases, as it allows for the reduction of complexity in modeling and analysis.

To formalize this, we introduce the empirical probability distribution of particles associated with the system defined by \eqref{nonlin_dynamics}. The empirical measure is defined as follows:
\begin{equation*}
    \mu^N(t,x) = \frac{1}{N} \sum_{j=1}^N \delta (x - x_j(t)),
\end{equation*}
where \(\delta(\cdot)\) represents the Dirac measure. This formulation captures the density of agents at position \(x\) at time \(t\) by summing contributions from all agents \(j\) in the system, normalized by the total number of agents \(N\).

We assume sufficient regularity conditions on the interaction kernel \(P(x, y)\) to ensure that the particles remain confined within a fixed compact domain. This is crucial as it guarantees that the density function behaves well and converges appropriately as \(N\) approaches infinity.

Starting from system \eqref{nonlin_dynamics} with control equals to \eqref{eq:u_feedback} applied to every agent (that is $B = \mathbbm{1}_{N\times 1}$), we derive the mean-field limit.
As \(N \to \infty\), the dynamics of the empirical measure converge to a nonlinear partial differential equation governing the evolution of the density \(\mu(t,x)\). Specifically, we obtain:
\begin{equation}\label{eq:mf_linear}
	\begin{split}
		\partial_t \mu(t,x) &= - \nabla_x \cdot \left[\mu(t,x) \Bigl(\mathcal{P}[\mu](t,x) + k(m_1[\mu](t)-c) \Bigr)\right],
		\\
		\mu(0,x) &= \mu^0(x),
	\end{split}
\end{equation}
where $\mathcal{P}$ is a non-local operator given by
\begin{equation*}
    \mathcal{P}[\mu](t,x) = \int_\mathbb{R} P(x,y) (y-x) \mu(t,y) dy,
\end{equation*}
and \(m_1[\mu]\) is the first moment (mean) of the probability density \(\mu(t,x)\), defined as:
\begin{align*}
	m_1[\mu](t) = \int_{\mathbb{R}} x \mu(t,x) \, dx.
\end{align*}
Equation \eqref{eq:mf_linear} highlights the interplay between local interactions and the collective dynamics of the ensemble, and we gain insights into how the properties of the original multi-agent system can be characterized in terms of macroscopic quantities, which may simplify both analysis and control design in practical applications.

\section{Leader-follower model} \label{sec:leader_follower}

In our exploration of multi-agent systems, we seek to balance the control exerted at the micro level over individual agents with the broader dynamics observed in a mean-field framework. To achieve this, we propose a leader-follower model that enables selective control of a subset of agents, allowing for a more nuanced interaction among the agents while ensuring effective collective behavior.

We define our system with a total of $N$ agents, comprising $N^F$ followers and $N^L$ leaders, such that $N^F + N^L = N$ with $N^F \gg N^L$. Importantly, the followers and leaders are distinct groups, which we denote by the sets $I^F$ for followers and $J^L$ for leaders, satisfying $I^F \cap J^L = \emptyset$.

To capture the interactions among agents, we assign weights to each group of agents, given by $\omega^F = \frac{\rho^F}{N^F}$ for followers and $\omega^L = \frac{\rho^L}{N^L}$, where the total fraction of agents must satisfy $\rho^F + \rho^L = 1$ with $\rho^F, \rho^L > 0$. This weighting allows us to scale interactions appropriately, especially when the number of followers is much larger than the number of leaders.

Such dynamics can be expressed as:
\begin{equation}\label{eq:leader_follower_micro}
    \begin{split}
	& \forall \  i \in I^F:
	\\
	&\dot{x}_i = \omega^F \sum_{j \in I^F} P(x_i, x_j)(x_j - x_i) + \omega^L \sum_{j \in J^L} P(x_i, x_j)(x_j - x_i),
	\\ \ \\
	&\forall \  i \in J^L:
	\\
	&\dot{x}_i = \omega^F \sum_{j \in I^F} P(x_i, x_j)(x_j - x_i) + \omega^L \sum_{j \in J^L} P(x_i, x_j)(x_j - x_i)
	\\
	& \qquad + k\left(\omega^F \sum_{j \in I^F} (x_j -c) + \omega^L \sum_{j \in J^L} (x_j - c)\right).
     \end{split}
\end{equation}
The first equation captures the influence of both the followers and the leaders on the movement of each follower. In the second one, the additional term involving the control parameter $k$ allows the leaders to adjust their behavior based on the average positions of both followers and other leaders, aiming for a target value $c$.

\begin{remark}
        Choosing $\omega^F = \omega^L = \frac{1}{N}$ and 
        \begin{equation*}
            b_i =
            \begin{cases}
                1 \quad \text{if } \ i\in J^L
                \\
                0 \quad \text{if } \ i\in I^F
            \end{cases}
        \end{equation*}
        we recover the original dynamics \eqref{nonlin_dynamics} with feedback control $u$ as in \eqref{eq:u_feedback}.
\end{remark}
To analyze the collective dynamics of the system, we shift to a mean-field perspective. We define the density of leaders as 
\begin{equation*}
	\eta^{N^L}(t,y) = \omega^L \sum_{j \in J^L} \delta(y - x_j(t)),
\end{equation*}
and the density of followers as 
\begin{equation*}
	\nu^{N^F}(t,x) = \omega^F \sum_{j \in I^F} \delta(x - x_j(t)).
\end{equation*}
As $N^F$ approaches infinity, the density of followers $\nu(t,x)$ emerges as a continuous distribution that captures the collective behavior of the followers.

In the mean-field limit, the evolution of the followers density and leaders is described by the equations:
\begin{equation}\label{eq:leader_follower_mf}
\begin{split}
	& \partial_t \nu(t,x) = - \nabla_x \cdot \left[ \nu(t,x) \left( \int_{\mathbb{R}} P(x,y)(y-x)\nu(t,y) \, dy \right) \right]
	\\
	& \hspace{1.8cm} - \nabla_x \cdot \left[ \nu(t,x) \left( \int_{\mathbb{R}} P(x,y)(y-x)\eta^{N^L}(t,y) \, dy \right) \right],
	\\ \ \\
	&\forall \  i \in J^L:
	\\
	& \dot{x}_i = \int_{\mathbb{R}} P(x_i,y)(y-x_i)\nu(t,y) \, dy + \omega^L \sum_{j \in J^L} P(x_i, x_j)(x_j - x_i)
	\\
	& \qquad + k\left( m_1[\nu] - \rho^F c + \omega^L \sum_{j \in J^L} (x_j - c)\right),
\end{split}
\end{equation}
where $m_1[\nu]$ is the average position of the followers, defined as 
\begin{equation*}
	m_1[\nu](t) = \int_{\mathbb{R}} x \, \nu(t,x) \, dx.
\end{equation*}
This framework integrates both the micro-level interactions of agents and the mean-field approximations, offering a structured approach to controlling a selective subset of agents while effectively managing the dynamics of the overall system.

\section{Lyapunov stability analysis}\label{sec:lyap}

In this section, we investigate the stability of our system using Lyapunov's direct method. The central idea is to prove that a Lyapunov function decreases over time, indicating that the system converges to a desired state or equilibrium.

\subsection{Full control Lyapunov decay}

We will focus first on the case where control is applied to every agent, hence we consider the original controlled dynamics \eqref{nonlin_dynamics} with $b_i=1$ for all $i\in\{1,\ldots,N\}$ and with control $u$ equal to \eqref{eq:u_feedback} choosing \( k < 0.\)

We define the Lyapunov function as follows:
\begin{equation*}
	\mathcal{L}(t) = \frac{1}{N} \sum_{i=1}^N (x_i(t) - c)^2,
\end{equation*}
where $c$ is the target state that the agents are attempting to reach. The intuition behind this choice is that $\mathcal{L}(t)$ measures the average squared deviation of the agents' positions from the desired target $c$.

For the rest of the article, we consider this additional assumption.

\begin{ass}\label{ass:P_non_negative}
	We assume the interaction kernel $P$ to be nonnegative and to be strictly positive across all pairs of agents $x\neq y$.
\end{ass}

\begin{prop}[Lyapunov function decay: Microscopic full control]\label{prop:micro_full_control}
If Assumption \ref{ass:P_non_negative} holds and the control \eqref{eq:u_feedback} with control parameter $k<0$ is applied uniformly to all agents, then the Lyapunov function $\mathcal{L}(t)$ satisfies the following inequality:
\begin{equation*}
	\frac{d}{dt}{\mathcal{L}}(t) < 0
\end{equation*}
for all $t$ with $x(t)\neq x_c$.
\end{prop}

\begin{proof}
    To prove this, we begin by differentiating the Lyapunov function with respect to time:
\begin{align*}
	\frac{d}{dt}{\mathcal{L}}(t) &= \frac{2}{N} \sum_{i=1}^N (x_i(t) - c) \dot{x}_i(t).
\end{align*}
Next, we substitute $\dot{x}_i(t)$ from the controlled dynamics \eqref{nonlin_dynamics}, with the control $u$ given by \eqref{eq:u_feedback}:
\begin{align*}
	\frac{d}{dt}{\mathcal{L}}(t) &= \frac{2}{N} \sum_{i=1}^N (x_i - c) \left( \frac{1}{N} \sum_{j=1}^N P(x_i, x_j)(x_j - x_i) + k \left( \frac{1}{N} \sum_{j=1}^N x_j - c \right) \right) \\
	&= \frac{2}{N^2} \sum_{i=1}^N \sum_{j=1}^N (x_i - c) P(x_i, x_j)(x_j - x_i) + \frac{2k}{N^2} \sum_{i=1}^N \sum_{j=1}^N (x_i - c)(x_j - c).
\end{align*}
We now simplify the two terms using the symmetry of $P(x_i, x_j)$ and swapping the indices:
\begin{align*}
	\frac{d}{dt}{\mathcal{L}}(t) &= \frac{1}{N^2} \sum_{i=1}^N  \sum_{j=1}^N (x_i - c) P(x_i, x_j)(x_j - x_i)
 \\
 & \qquad - \frac{1}{N^2} \sum_{i=1}^N  \sum_{j=1}^N (x_j - c) P(x_i, x_j)(x_j - x_i)
 + \frac{2k}{N^2} \Biggl( \sum_{i=1}^N (x_i - c)\Biggr)^2
 \\
 & = -\frac{1}{N^2} \sum_{i=1}^N  \sum_{j=1}^N P(x_i, x_j)(x_j - x_i)^2
+ \frac{2k}{N^2} \Biggl( \sum_{i=1}^N (x_i - c)\Biggr)^2.
\end{align*}
Since $ P(x_i, x_j) \geq 0 $ according to Assumption \ref{ass:P_non_negative}, and since $ (x_j - x_i)^2 \geq 0 $, the first term in the previous calculation is non-positive. It is zero if and only if all agents share the same position, i.e. $x_i=x_j$ for all $i,j\in\{1,\ldots,N\}$.
Furthermore, the second term is non-positive as $k<0$. 
We conclude that:
\begin{equation*}
	\frac{d}{dt}{\mathcal{L}}(t) \leq 0
\end{equation*}
with $\frac{d}{dt}{\mathcal{L}}(t) = 0$ if and only if both terms in the previous calculation are zero, which is equivalent to $x(t)=x_c$. Therefore, we have $\frac{d}{dt}{\mathcal{L}}(t) <0$ for all $t$ with $x(t)\neq x_c$, proving the asymptotic stability of the system under full control.
\end{proof}

We established a control strategy that effectively drives the agents towards the desired equilibrium state $x_c$.

We consider now the mean-field setting, and we aim to establish the decay of a Lyapunov function associated with the evolution of the probability density $\mu(t,x)$, which describes the distribution of agents' positions in the limit as the number of agents tends to infinity.

We define the mean-field Lyapunov function as:

\begin{equation*}
	\mathcal{L}^{\mu}(t) = \int_{\mathbb{R}} |x - c|^2 \, \mu(t,x) \, dx,
\end{equation*}
where $c$ is the target state, and $\mu(t,x)$ is the probability density of the agents at time $t$. 
The dynamics of $\mu(t,x)$ are governed by the non-local PDE \eqref{eq:mf_linear}.
\begin{prop}[Lyapunov function decay: Mean-field full control]\label{prop:mf_full_control}
If Assumption \ref{ass:P_non_negative} holds and the mean-field dynamics \eqref{eq:mf_linear} are controlled via the feedback term $k(m_1[\mu](t)-c)$ with $k<0$, then the mean-field Lyapunov function $\mathcal{L}^{\mu}(t)$ satisfies the following inequality:
\begin{equation*}
    \frac{d}{dt} \mathcal{L}^{\mu}(t) < 0.
\end{equation*}
\end{prop}

\begin{proof}
We differentiate $\mathcal{L}^{\mu}(t)$ with respect to time:
\begin{equation*}
    \frac{d}{dt} \mathcal{L}^{\mu}(t) = \frac{d}{dt} \int_{\mathbb{R}} |x - c|^2 \mu(t,x) \, dx = \int_{\mathbb{R}} |x - c|^2 \partial_t \mu(t,x) \, dx.
\end{equation*}
Substituting the continuity equation from the mean-field dynamics \eqref{eq:mf_linear} we get:
\begin{align*}
    \frac{d}{dt} \mathcal{L}^{\mu}(t) &= - \int_{\mathbb{R}} |x - c|^2 \nabla_x \cdot \left[\mu(t,x) \Bigl(\mathcal{P}[\mu](t,x) + k(m_1[\mu](t) - c)\Bigr)\right] \, dx.
\end{align*}
Using integration by parts and assuming that $\mu(t,x)$ vanishes at the boundary, this simplifies to:
\begin{align*}
    \frac{d}{dt} \mathcal{L}^{\mu}(t) &= 2 \int_{\mathbb{R}} (x - c) \mu(t,x) \Bigl(\mathcal{P}[\mu](t,x) + k(m_1[\mu](t) - c)\Bigr) \, dx.
\end{align*}
Now, let’s analyze the two terms separately. For the non-local interaction term $\mathcal{P}[\mu](t,x)$, we have:
\begin{align*}
    \int_{\mathbb{R}} (x - c) \mu(t,x) \mathcal{P}[\mu](t,x) \, dx &= \int_{\mathbb{R}} \int_{\mathbb{R}} (x - c) P(x,y) (y - x) \mu(t,x) \mu(t,y) \, dx \, dy.
\end{align*}
By symmetry of $P(x,y)$ and swapping the indices, it follows that:
\begin{align*}
    &\int_{\mathbb{R}} \int_{\mathbb{R}} (x - c) P(x,y) (y - x) \mu(t,x) \mu(t,y) \, dx \, dy \\
    &= - \frac{1}{2}\int_{\mathbb{R}} \int_{\mathbb{R}} P(x,y) (y - x)^2 \mu(t,x) \mu(t,y) \, dx \, dy,
\end{align*}
which is negative as $\mu(t,\cdot)$ is a density function and $P(x,y)>0$ for $x\neq y$ according to Assumption \ref{ass:P_non_negative}.
This shows that the contribution from the interaction term is negative.
For the control term involving the feedback $k(m_1[\mu](t) - c)$, we have:
\begin{align*}
    k \int_{\mathbb{R}} (x - c) \mu(t,x) (m_1[\mu](t) - c) \, dx &= k(m_1[\mu](t) - c) \int_{\mathbb{R}} (x - c) \mu(t,x) \, dx \\
    &= k(m_1[\mu](t) - c) \left(m_1[\mu](t) - c\right) \\
    &= k(m_1[\mu](t) - c)^2.
\end{align*}
This quantity is non-positive as $k<0$ and it equals zero if and only if $m_1[\mu](t)=c$. 
Thus, combining both terms, we conclude that:
\begin{equation*}
    \frac{d}{dt} \mathcal{L}^{\mu}(t) < 0,
\end{equation*}
which shows that the Lyapunov function $\mathcal{L}^{\mu}(t)$ is decreasing over time.
\end{proof}

\subsection{Sparse control Lyapunov decay}

In this section, we extend our analysis of the Lyapunov function decay to the leader-follower model under sparse control, please refer to section \ref{sec:leader_follower}. Specifically, we will show that the decay of the Lyapunov function for the followers occurs under certain conditions.

We start with the microscopic case, then we define the Lyapunov functions for followers and leaders as follows:
\begin{align*}
    \mathcal{L}^F(t) &= \omega^F \sum_{i \in I^F} (x_i(t) - c)^2, \\
    \mathcal{L}^L(t) &= \omega^L \sum_{i \in J^L} (x_i(t) - c)^2.
\end{align*}
We recall that here, $\omega^F$ and $\omega^L$ are the weights associated with the followers and leaders, respectively, while $c$ is the target state.
The dynamics of leaders and followers are given by \eqref{eq:leader_follower_micro}.

\begin{prop}[Lyapunov Function Decay: Microscopic Sparse Control] \label{prop:micro_sparse_control}
Let Assumption \ref{ass:P_non_negative} hold. If \( P(x_i, x_j) = \bar{p} \) (i.e., \(P\) is constant across all pairs of agents and the dynamics are linear), and if the control gain $k<0$ satisfies the condition
\begin{equation}\label{eq:cond_decay}
    |k| > 2 \bar{p},
\end{equation}
then the Lyapunov functions for the followers and leaders satisfy:
\[
\frac{d}{dt} \left( \mathcal{L}^F(t) + \mathcal{L}^L(t) \right) \leq \beta \left( \mathcal{L}^F(t) + \mathcal{L}^L(t) \right),
\]
for some $\beta < 0$.
\end{prop}

\begin{proof}
We begin by analyzing the time derivative of the followers' Lyapunov function $\mathcal{L}^F(t)$:
\[
\frac{d}{dt} \mathcal{L}^F(t) = \frac{d}{dt} \left( \omega^F \sum_{i \in I^F} (x_i - c)^2 \right) = 2 \omega^F \sum_{i \in I^F} (x_i - c) \dot{x}_i.
\]
Substituting the dynamics of the followers:
\[
\dot{x}_i = \omega^F \sum_{j \in I^F} P(x_i, x_j)(x_j - x_i) + \omega^L \sum_{j \in J^L} P(x_i, x_j)(x_j - x_i),
\]
we get:
\begin{align*}
    \frac{d}{dt} \mathcal{L}^F(t) &= 2 \omega^F \sum_{i \in I^F} (x_i - c) \left( \omega^F \sum_{j \in I^F} P(x_i, x_j)(x_j - x_i) \right)
    \\
    & \qquad + 2 \omega^F \sum_{i \in I^F} (x_i - c) \left(\omega^L \sum_{j \in J^L} P(x_i, x_j)(x_j - x_i) \right).
\end{align*}
Expanding this expression, we define two terms:
\[
Q_1 = 2 (\omega^F)^2 \sum_{i \in I^F} \sum_{j \in I^F} (x_i - c) P(x_i, x_j)(x_j - x_i),
\]
\[
Q_2 = 2 \omega^F \omega^L \sum_{i \in I^F} \sum_{j \in J^L} (x_i - c) P(x_i, x_j)(x_j - x_i).
\]
Thus, the time derivative of $\mathcal{L}^F(t)$ can be written as:
\[
\frac{d}{dt} \mathcal{L}^F(t) = Q_1 + Q_2.
\]
\textbf{Analysis of $Q_1$:}
Using the same steps as in the proof of the microscopic full control case (please refer to Proposition \ref{prop:micro_full_control}), we can swap the indices and use the symmetry of $P(x_i, x_j)$ to show that this term is non-positive:
		\[
		Q_1 = - (\omega^{F})^2 \sum_{i \in I^F} \sum_{j \in I^F} P(x_i, x_j)(x_j - x_i)^2 \leq 0.
		\]
\\
\textbf{Analysis of $Q_2$:}
Substituting $P(x_i, x_j) = \bar p$, with few computations, we rewrite $Q_2$:
\begin{align*}
    Q_2 &= 2 \bar p \, \omega^F \omega^L \sum_{i \in I^F} (x_i - c) \sum_{j \in J^L} (x_j - c)
    \, - \, 2 \bar p \, \omega^F \omega^L N^L\sum_{i \in I^F} (x_i - c)^2 .
\end{align*}
Since the second term is non-positive, we obtain:
\[ Q_2 \leq 2 \bar p \, \omega^F \omega^L \sum_{i \in I^F} (x_i - c) \sum_{j \in J^L} (x_j - c).\]

Next, we analyze the leaders' Lyapunov function $\mathcal{L}^L(t)$. Its time derivative is:
\[
\frac{d}{dt} \mathcal{L}^L(t) = \frac{d}{dt} \left( \omega^L \sum_{i \in J^L} (x_i - c)^2 \right) = 2 \omega^L \sum_{i \in J^L} (x_i - c) \dot{x}_i.
\]
Substituting the dynamics of the leaders:
\begin{align*}
    \dot{x}_i &= \omega^F \sum_{j \in I^F} P(x_i, x_j)(x_j - x_i) + \omega^L \sum_{j \in J^L} P(x_i, x_j)(x_j - x_i)
    \\
    & \qquad + k \left( \omega^F \sum_{j \in I^F} (x_j - c) + \omega^L \sum_{j \in J^L} (x_j - c) \right),
\end{align*}
we get:
\[
\frac{d}{dt} \mathcal{L}^L(t) = Q_3 + Q_4 + Q_5 + Q_6,
\]
where:
\begin{align*}
    Q_3 &= 2 (\omega^L)^2 \sum_{i \in J^L} \sum_{j \in J^L} (x_i - c) P(x_i, x_j)(x_j - x_i),
\\
Q_4 &= 2 (\omega^L)^2 k \sum_{i \in J^L} (x_i - c) \sum_{j \in J^L} (x_j - c),
\\
Q_5 &= 2 \omega^L \omega^F \sum_{i \in J^L} \sum_{j \in I^F} (x_i - c) P(x_i, x_j)(x_j - x_i),
\\
Q_6 &= 2 \omega^L \omega^F k \sum_{i \in J^L} (x_i - c) \sum_{j \in I^F} (x_j - c).
\end{align*}
\textbf{Analysis of $Q_3$:}
We observe that $Q_3$ has the same form as $Q_1$.
\\
\textbf{Analysis of $Q_4$:}
This term involves the feedback control applied to the leaders. Since $k<0$, it can be rewritten as:
\[
Q_4 = - 2 (\omega^L)^2 |k| \left( \sum_{i \in J^L} (x_i - c) \right)^2\leq 0.
\]
\\
\textbf{Analysis of $Q_5$:}
We observe that $Q_5$ has the same form as $Q_2$.
\\
\textbf{Analysis of $Q_6$:}
Since $k<0$, this term can be rewritten as:
\[
Q_6 = -2 \omega^F \omega^L |k| \sum_{i \in I^F} (x_i - c) \sum_{j \in J^L} (x_j - c).
\]
\\
The derivative of the total Lyapunov function is:
\[
\frac{d}{dt} \left( \mathcal{L}^F + \mathcal{L}^L \right) = Q_1 + Q_2 + Q_3 + Q_4 + Q_5 + Q_6.
\]
Substituting results for $Q_1, Q_3, Q_4$:
\[
Q_1 + Q_3 + Q_4 \leq 0,
\]
and combining $Q_2, Q_5, Q_6$:
\[
Q_2 + Q_5 + Q_6\leq \left( 4 \bar{p} - 2 |k| \right) \left( \omega^F \sum_{i \in I^F} (x_i - c) \right) \left( \omega^L \sum_{j \in J^L} (x_j - c) \right),
\]
we have:
\begin{align*}
    \frac{d}{dt} \left( \mathcal{L}^F + \mathcal{L}^L \right)
    &\leq \left( 4 \bar{p} - 2 |k| \right) \left( \omega^F \sum_{i \in I^F} (x_i - c) \right) \left( \omega^L \sum_{j \in J^L} (x_j - c) \right).
\end{align*}
Using the Cauchy-Schwarz inequality, we can bound the cross terms as:
\[
\left( \sum_{i \in I^F} (x_i - c) \right) \left( \sum_{j \in J^L} (x_j - c) \right)
\leq \sqrt{\sum_{i \in I^F} (x_i - c)^2} \sqrt{\sum_{j \in J^L} (x_j - c)^2}.
\]
Using the inequality \(\sqrt{\mathcal{L}^F \mathcal{L}^L} \leq \frac{1}{2} \left( \mathcal{L}^F + \mathcal{L}^L \right)\), we have:
\[
\frac{d}{dt} \left( \mathcal{L}^F + \mathcal{L}^L \right) 
\leq \left( 4 \bar{p} - 2 |k| \right) \frac{\sqrt{\omega^F \omega^L}}{2} \left( \mathcal{L}^F + \mathcal{L}^L \right).
\]
Defining \(\beta = \left( 4 \bar{p} - 2 |k| \right) \frac{\sqrt{\omega^F \omega^L}}{2}\), the decay condition \(\beta < 0\) is satisfied when:
\[
|k| > 2 \bar{p}.
\]
Thus, under the condition \(k < -2 \bar{p}\), we ensure that:
\[
\frac{d}{dt} \left( \mathcal{L}^F + \mathcal{L}^L \right) \leq \beta \left( \mathcal{L}^F + \mathcal{L}^L \right),
\]
where \(\beta < 0\), completing the proof.
\end{proof}

For the sparse case, we observe that a result of Lyapunov function decay holds only when the dynamics are linear, specifically when \( P(x_i, x_j) = \bar{p} \) is constant across all pairs of agents. In this case, the decay condition depends on the control gain satisfying \( |k| > 2 \bar{p} \). This condition ensures that the total Lyapunov function decays over time, ultimately leading to the stabilization of the system. However, for non-linear dynamics, such as when \( P(x_i, x_j) \) is not constant, the analysis and conditions for decay may change, and the result does not directly apply.

We extend now our analysis of the Lyapunov function decay to the leader-follower model under sparse control, considering a mean-field representation for the followers while keeping the leaders in a microscopic framework. This builds upon the analysis presented in Section~\ref{sec:leader_follower}.
Specifically, we aim to demonstrate that the decay of the Lyapunov functions for the system occurs under certain conditions. Here, the followers' dynamics are described by a density $\nu(t, x)$ in the mean-field limit, while the leaders' dynamics are given in terms of individual agent positions $x_i(t)$ for $i \in J^L$.

We define the Lyapunov functions for the mean-field followers and the microscopic leaders as follows:
\begin{align*}
    \mathcal{L}^\nu(t) &= \int_{\mathbb{R}} |x - c|^2 \nu(t,x) \, dx, \\
    \mathcal{L}^\eta(t) &= \omega^L \sum_{i \in J^L} (x_i(t) - c)^2.
\end{align*}
Here, $\omega^L$ represent the weight associated with the leaders, and $c$ is the target state for both populations.
The dynamics of the system are governed by the mean-field equations for the followers and the microscopic equations for the leaders as in Eq. \eqref{eq:leader_follower_mf}. 
In the following, we analyze the time evolution of the Lyapunov functions $\mathcal{L}^\nu(t)$ and $\mathcal{L}^\eta(t)$, showing that their combined decay is governed by the interactions between the mean-field followers and microscopic leaders, under appropriate conditions on the parameters of the system.

\begin{prop}[Lyapunov Function Decay: Mean-field Sparse Control]\label{prop:LF_MF}
Let Assumption \ref{ass:P_non_negative} hold. If \( P(x, y) = \bar{p} \) (i.e., \(P\) is constant across all pairs of agents and the dynamics are linear), and if the control gain $k<0$ satisfies the condition
\[
|k| > 2 \bar{p},
\]
then the Lyapunov functions for the mean-field followers and microscopic leaders satisfy:
\[
\frac{d}{dt} \left( \mathcal{L}^\nu(t) + \mathcal{L}^\eta(t) \right) \leq \beta \left( \mathcal{L}^\nu(t) + \mathcal{L}^\eta(t) \right),
\]
for some $\beta < 0$.
\end{prop}

\begin{proof}
We begin the proof by analyzing the time derivative of the Lyapunov function for the followers' density, \(\mathcal{L}^\nu(t)\):
\[
\frac{d}{dt} \mathcal{L}^\nu(t) = \int_{\mathbb{R}} |x - c|^2 \cdot \partial_t \nu(t, x) \, dx.
\]
Substituting the continuity equation governing the evolution of \(\nu(t,x)\) from \eqref{eq:leader_follower_mf}, and integrating by parts, assuming that \(\nu(t,x)\) vanishes at the boundary, this simplifies to:
\[
\frac{d}{dt} \mathcal{L}^\nu(t) = \tilde{Q}_1 + \tilde{Q}_2,
\]
where we define:
\begin{align*}
\tilde{Q}_1 &= 2 \int_{\mathbb{R}} \int_{\mathbb{R}}  (x - c)   P(x, y)(y - x) \nu(t, x) \nu(t, y)  \, dx \, dy,
\\
\tilde{Q}_2 &= 2 \int_{\mathbb{R}} \int_{\mathbb{R}} (x - c)   P(x, y)(y - x)\nu(t, x) \eta^{N^L}(t, y) \, dx  \, dy.
\end{align*}
\textbf{Analysis of \(\tilde{Q}_1\):}  
Exactly as in the proof of Proposition \ref{prop:mf_full_control}, we exploit the symmetry of \(P(x, y)\) and perform a swapping of indices, noting that:
\[
\tilde{Q}_1 \leq 0
\]
\textbf{Analysis of \(\tilde{Q}_2\):}  
Assuming $P(x,y) = \bar{p}$ and expanding the term \((y - x) = (y - c) - (x - c)\), we have:
\begin{align*}
\tilde{Q}_2 &= 2 \bar{p} \int_{\mathbb{R}} \int_{\mathbb{R}} (x - c) \Big[(y - c) - (x - c)\Big] \nu(t, x) \eta^{N^L}(t, y) \, dx \, dy
\\
&=2 \bar{p} \int_{\mathbb{R}} (x - c) \nu(t, x) \, dx \int_{\mathbb{R}} (y - c) \eta^{N^L}(t, y) \, dy 
\\
& \qquad - 2 \bar{p} \int_{\mathbb{R}} \int_{\mathbb{R}} |x - c|^2\nu(t, x) \eta^{N^L}(t, y) \, dx \, dy.
\end{align*}
Since the second term is non-positive, we can write:
\[
\tilde{Q}_2 \leq 2 \bar{p} \,  (m_1[\nu](t) - \rho^F \, c) \,(m_1[\eta^{N^L}](t) - \rho^L \, c).
\]
Thus, \(\tilde{Q}_2\) simplifies to a coupling term involving the average positions of the leaders and followers.

Next, we analyze the time derivative of the leaders' Lyapunov function \(\mathcal{L}^\eta(t)\). We begin by computing:
\[
\frac{d}{dt} \mathcal{L}^\eta(t) = 2 \omega^L \sum_{i \in J^L} (x_i(t) - c) \, \dot{x}_i(t).
\]
If we substitue the dynamics of the leaders from \eqref{eq:leader_follower_mf},
this leads to the following expression:
\begin{align*}
\frac{d}{dt} \mathcal{L}^\eta(t) = \, \tilde Q_3 + \tilde Q_4 +\tilde Q_5 +\tilde Q_6,
\end{align*}
where the terms are defined as follows:
\begin{align*}
\tilde Q_3 &= 2 (\omega^L)^2 \sum_{i \in J^L} \sum_{j \in J^L} (x_i - c) P(x_i, x_j)(x_j - x_i),
\\
\tilde Q_4 &= 2 (\omega^L)^2 k \sum_{i \in J^L} (x_i - c) \sum_{j \in J^L} (x_j - c),
\\
\tilde Q_5 &= 2 \omega^L \sum_{i \in J^L} (x_i - c) \int_{\mathbb{R}} P(x_i, y)(y - x_i) \nu(t,y) \, dy ,
\\
\tilde Q_6 &= 2 \omega^L k \sum_{i \in J^L} (x_i - c) \left( m_1[\nu] - \rho^F c\right).
\end{align*}
\textbf{Analysis of $\tilde Q_3$:}
We observe that this term is equal to $Q_3$ in the proof of Proposition \ref{prop:micro_sparse_control}, then
\[
\tilde Q_3 \leq 0.
\]
\\
\textbf{Analysis of $\tilde Q_4$:}
We also have that $\tilde Q_4 = Q_4$ from the proof of Proposition \ref{prop:micro_sparse_control}. We then have
\[\tilde Q_4 \leq 0\]
\\
\textbf{Analysis of $\tilde Q_5$:}
Since
\begin{equation*}
	\eta^{N^L}(t,y) = \omega^L \sum_{j \in J^L} \delta(y - x_j(t)),
\end{equation*}
$\tilde Q_5$ has the same form as $\tilde Q_2$.
\\
\textbf{Analysis of $\tilde Q_6$:}
Since $k<0$, this term can be rewritten as:
\[
\tilde Q_6 = -2 \omega^L |k| \sum_{i \in J^L} (x_i - c) \left( m_1[\nu] - \rho^F c\right).
\]
\\
The derivative of the total Lyapunov function is:
\[
\frac{d}{dt} \left( \mathcal{L}^\nu + \mathcal{L}^\eta \right) = \tilde Q_1 + \tilde Q_2 + \tilde Q_3 + \tilde Q_4 + \tilde Q_5 + \tilde Q_6.
\]
Substituting results for $\tilde Q_1, \tilde Q_3, \tilde Q_4$:
\[
\tilde Q_1 + \tilde Q_3 + \tilde Q_4 \leq 0,
\]
and combining $\tilde Q_2, \tilde Q_5, \tilde Q_6$:
\[
\tilde Q_2 + \tilde Q_5 + \tilde Q_6 \leq \left( 4 \bar{p} - 2 |k| \right) \,  (m_1[\nu](t) - \rho^F \, c) \,(m_1[\eta^{N^L}](t) - \rho^L \, c),
\]
we have:
\begin{align*}
    \frac{d}{dt} \left( \mathcal{L}^\nu + \mathcal{L}^\eta \right)
    &\leq \left( 4 \bar{p} - 2 |k| \right) \,  (m_1[\nu](t) - \rho^F \, c) \,(m_1[\eta^{N^L}](t) - \rho^L \, c).
\end{align*}
Using the Cauchy-Schwarz inequality, we can bound the cross terms as:
\[
(m_1[\nu](t) - \rho^F \, c) \, (m_1[\eta^{N^L}](t) - \rho^L \, c)
\leq \sqrt{\mathcal{L}^\nu \, \mathcal{L}^\eta } \leq \frac{1}{2}({\mathcal{L}^\nu + \mathcal{L}^\eta }).
\]
Thus, we can write:
\[
\frac{d}{dt} \left( \mathcal{L}^\nu + \mathcal{L}^\eta \right) 
\leq \left( 4 \bar{p} - 2 |k| \right) \frac{1}{2} \left( \mathcal{L}^\nu + \mathcal{L}^\eta \right).
\]
Defining \(\beta = \left( 4 \bar{p} - 2 |k| \right)/2\), the decay condition \(\beta < 0\) is satisfied when:
\[
|k| > 2 \bar{p}.
\]
Therefore, under the condition \(k < -2 \bar{p}\), we ensure that:
\[
\frac{d}{dt} \left( \mathcal{L}^\nu + \mathcal{L}^\eta \right) \leq \beta \left( \mathcal{L}^\nu + \mathcal{L}^\eta \right),
\]
where \(\beta < 0\). This completes the proof.
\end{proof}

\section{Numerical tests}\label{sec:numtests}

In this section, we present numerical experiments to validate the proposed theoretical results. In all simulations, we employ the following interaction function:  
\begin{equation*}
    P(x_i, x_j) = \frac{\bar{p}}{(1+ \Vert x_i - x_j \Vert^2)^2},
\end{equation*}  
where $\bar{p} = P(c,c) = 0.04$ represents the maximum attraction strength. The denominator ensures that interactions weaken as the distance between agents increases, leading to a decay in influence over long ranges.
This interaction kernel is reminiscent of the classical Cucker-Smale model \cite{cucker2007emergent}, which is typically employed in second-order models to describe velocity alignment in collective dynamics. However, in our case, we adapt this type of kernel within a first-order model, where it plays the role of an attraction mechanism rather than an alignment rule. This choice allows us to capture nonlocal interactions that drive agents toward consensus-like behavior while preserving the qualitative decay properties of the original Cucker-Smale formulation.

The initial time is always set to $t_0 = 0$, with a time step of $\Delta t = 0.01$ and a total simulation horizon of $T = 400$, ensuring a sufficiently long observation window.

The target state is set to $c = 1$, while the control parameter is chosen as $k = -0.1$. Notably, these parameter choices satisfy condition \eqref{eq:cond_decay}, ensuring the exponential decay of the Lyapunov functions in the leader-follower model.

The initial positions are drawn from a uniform distribution over the interval $[2,5]$ for both the microscopic and mean-field settings, providing a consistent initialization framework across different scales of the model.

We introduce now the numerical methods used to approximate the microscopic dynamics given by \eqref{nonlin_dynamics}-\eqref{eq:leader_follower_micro} and their mean-field counterparts \eqref{eq:mf_linear}-\eqref{eq:leader_follower_mf}.
For the numerical experiments, we discretize the microscopic dynamics using a forward Euler scheme with time step $\Delta t = 0.01$ over the time horizon $[t_0, T]$, and a total number of agents $N=50$.

To approximate the mean-field dynamics governed by \eqref{eq:mf_linear}-\eqref{eq:leader_follower_mf}, we adopt the Mean-Field Monte Carlo (MFMC) method introduced in \cite{albi2013binary}.
Specifically, for the approximation of \eqref{eq:mf_linear}, we consider a set of $\hat{N} = 10000$ particles sampled from the initial distribution $\mu^0(x)$. The mean-field evolution is then approximated through a particle-based approach, where the non-local integrals describing interactions are computed using finite sums over a randomly selected subset of $\hat{N}_s = 1000$ particles.

This Monte Carlo approach reduces the computational cost of evaluating interaction terms from $\mathcal{O}(\hat{N}^2)$ to $\mathcal{O}(\hat{N}_s \hat{N})$. In the limiting case where $\hat{N}_s = \hat{N}$, the method recovers the full microscopic system with $\hat{N}$ particles.

The same strategy is employed for the leader-follower model; however, in this case, the MFMC method is applied solely to the follower density. We remind that in \eqref{eq:leader_follower_mf}, only the followers are treated as a continuous density, while the leader remains a microscopic agent.
In the following simulations, the single leader ($N^L = 1$), is evolved microscopically, while we employ MFMC sampling $\hat{N}^F = 9999$ particles from the initial follower distribution of followers $\nu^0(x)$, and approximating the non-local integrals with finite sums over a randomly selected subset of $\hat{N}^F_s = 1000$ particles.

\begin{figure}[H]
	\centering
	\includegraphics[width=0.45\linewidth]{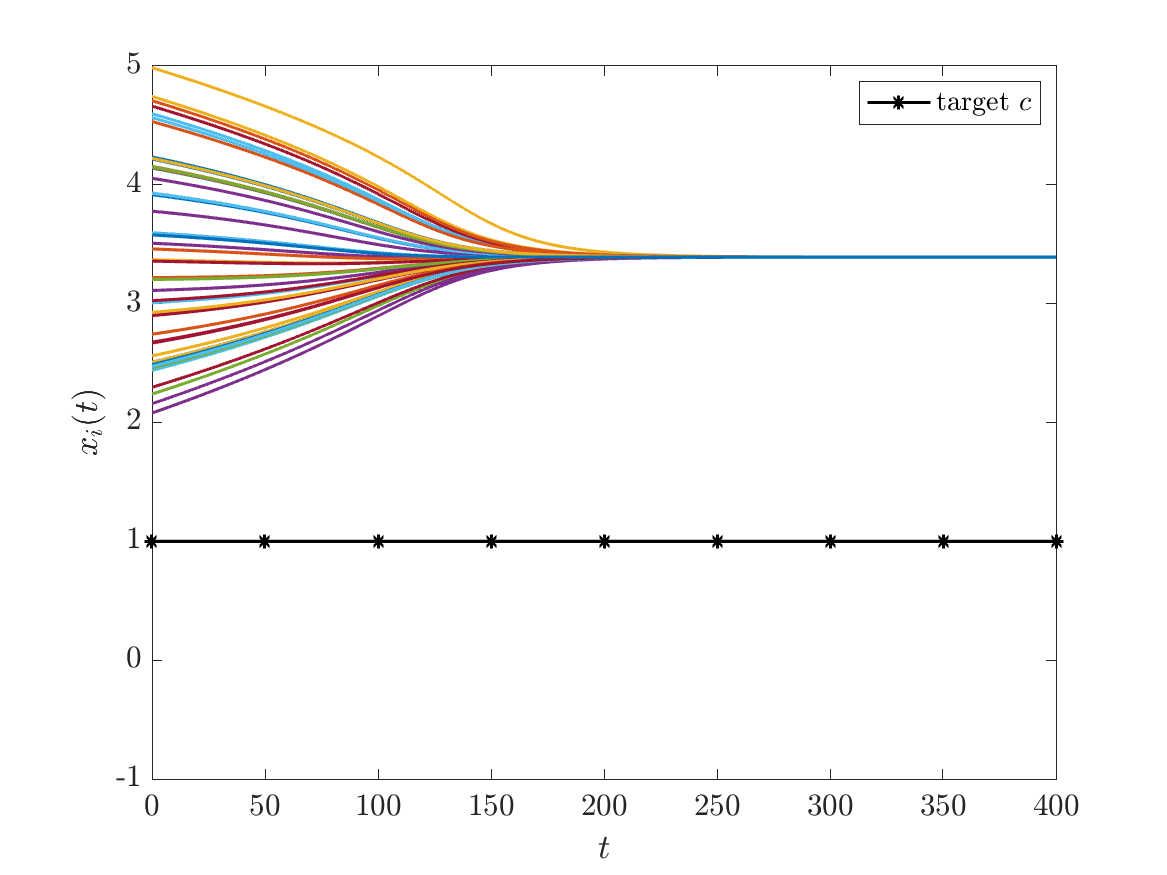}
	\includegraphics[width=0.45\linewidth]{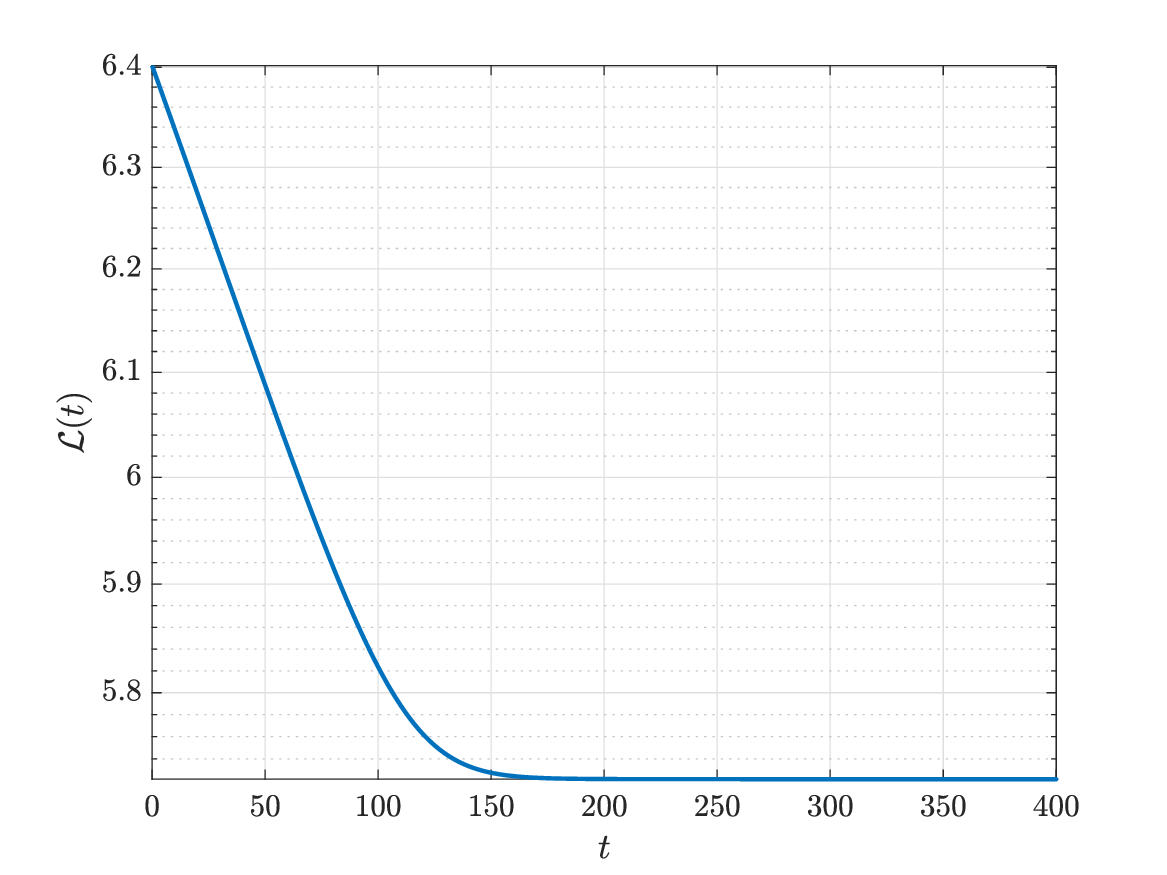}\\
	\includegraphics[width=0.45\linewidth]{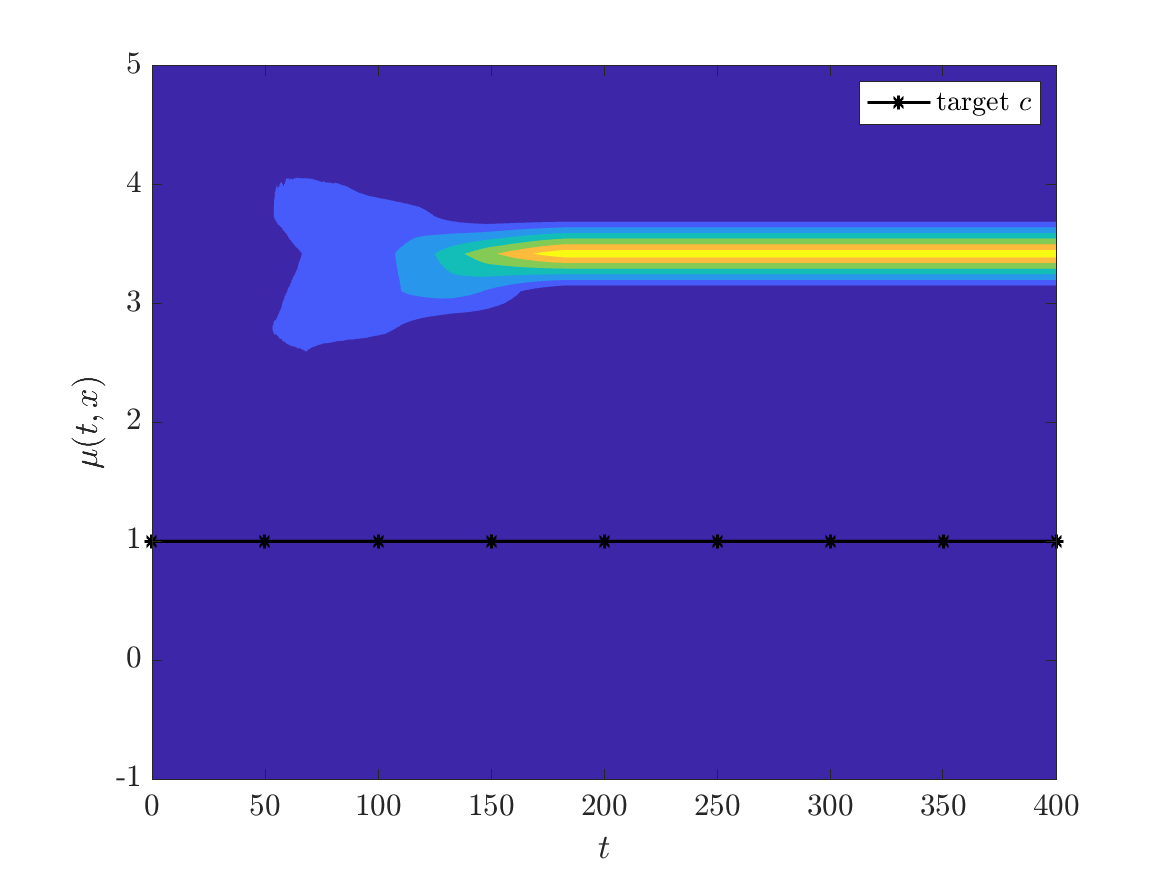}
	\includegraphics[width=0.45\linewidth]{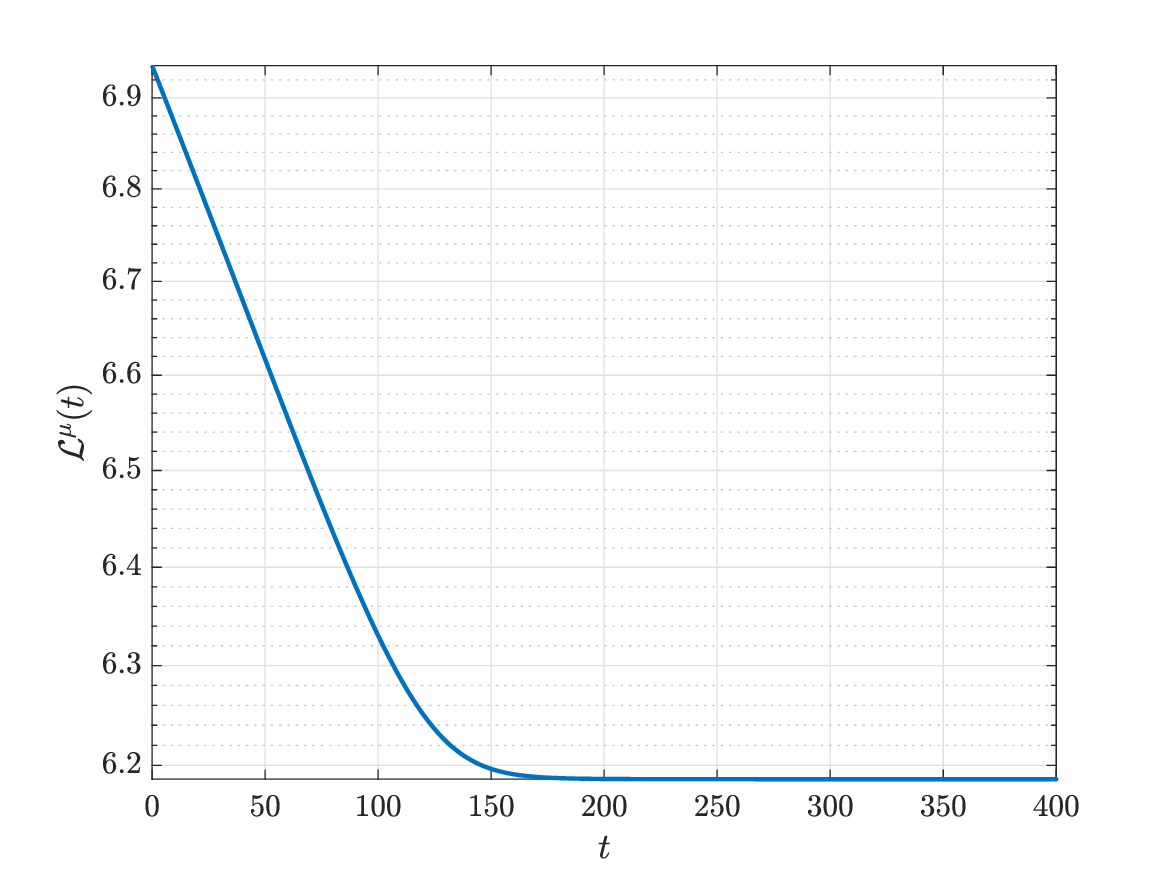}
	\caption{Evolution of the uncontrolled system. Left column: microscopic trajectories (top) and mean-field density (bottom). Right column: Lyapunov function decay in microscopic and mean-field settings.}
	\label{fig:uncontrolled}
\end{figure}

We first analyze the non-linear system in the absence of control. This means that in Eq.~\eqref{eq:dyn_contr}, we consider $b_i$ identically zero for every agent $i$. Figure \ref{fig:uncontrolled} shows that agents fail to reach the target state $c=1$ both for the microscopic and mean-field dynamics. The agents reach a consensus point given by their initial mean position. The Lyapunov function does not decrease sufficiently over time and remains approximately at the value $6$, indicating that stabilization requires a control strategy.

\begin{figure}[H]
	\centering
	\includegraphics[width=0.45\linewidth]{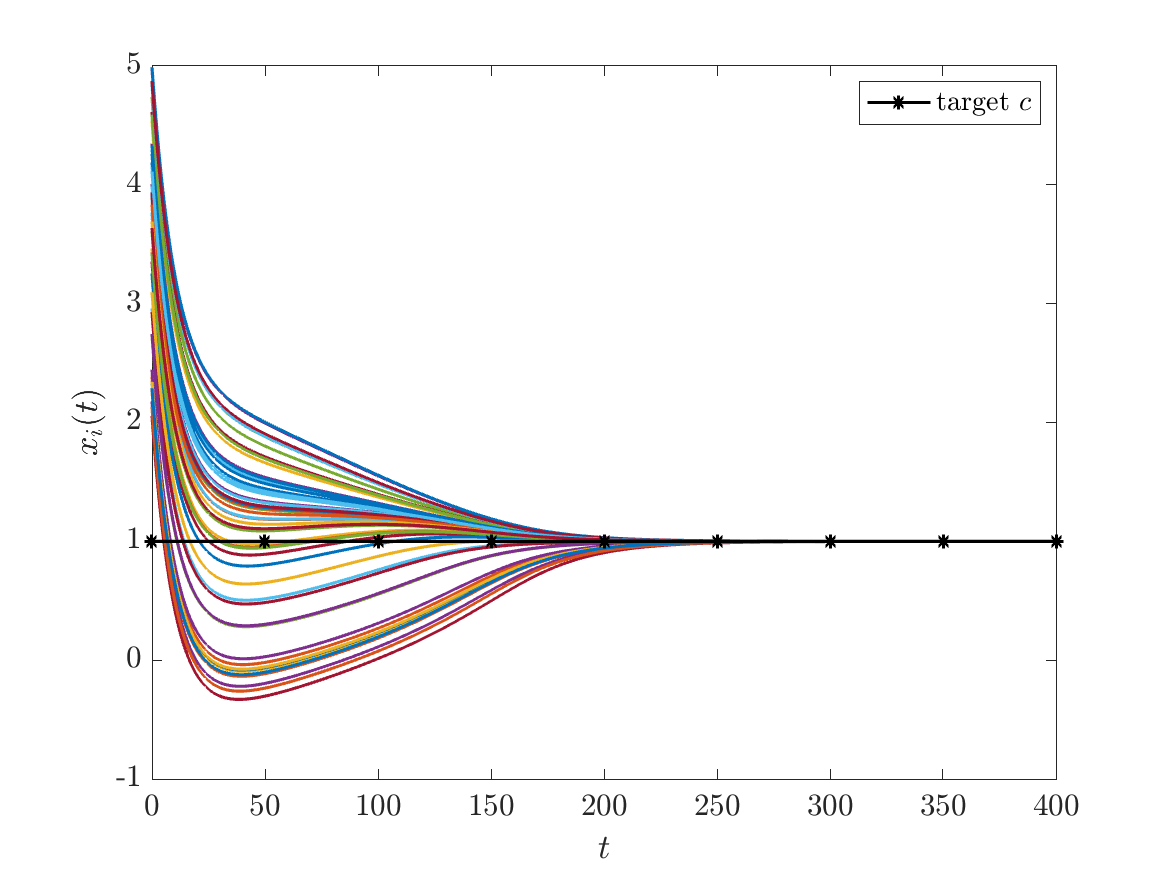}
	\includegraphics[width=0.45\linewidth]{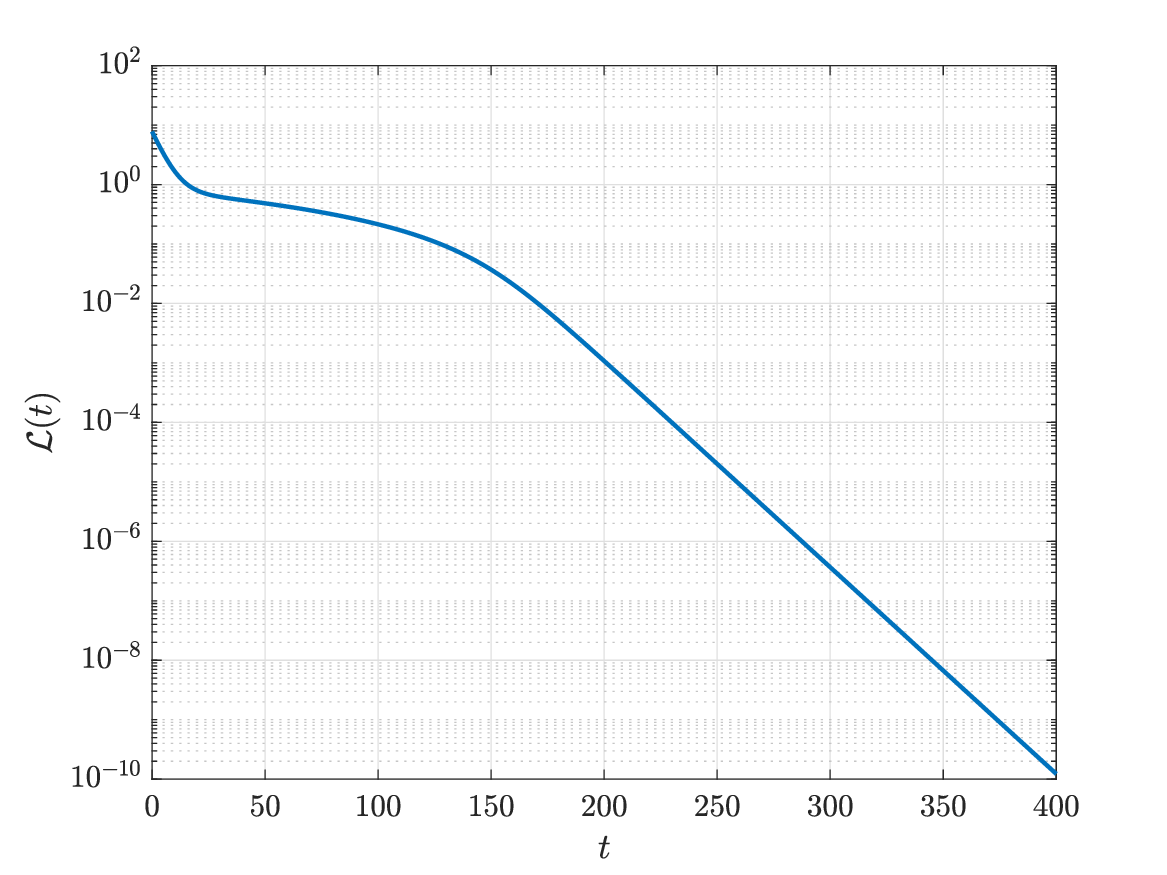}\\
	\includegraphics[width=0.45\linewidth]{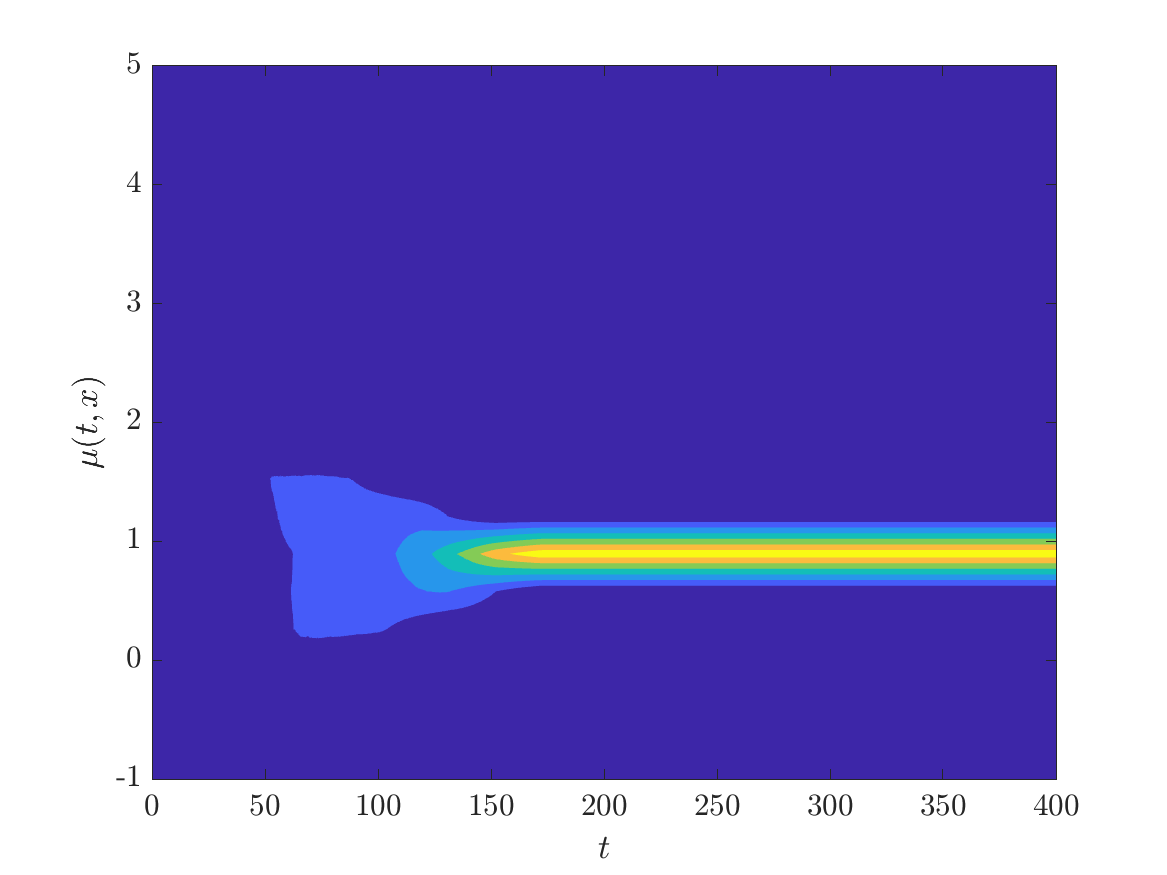}
	\includegraphics[width=0.45\linewidth]{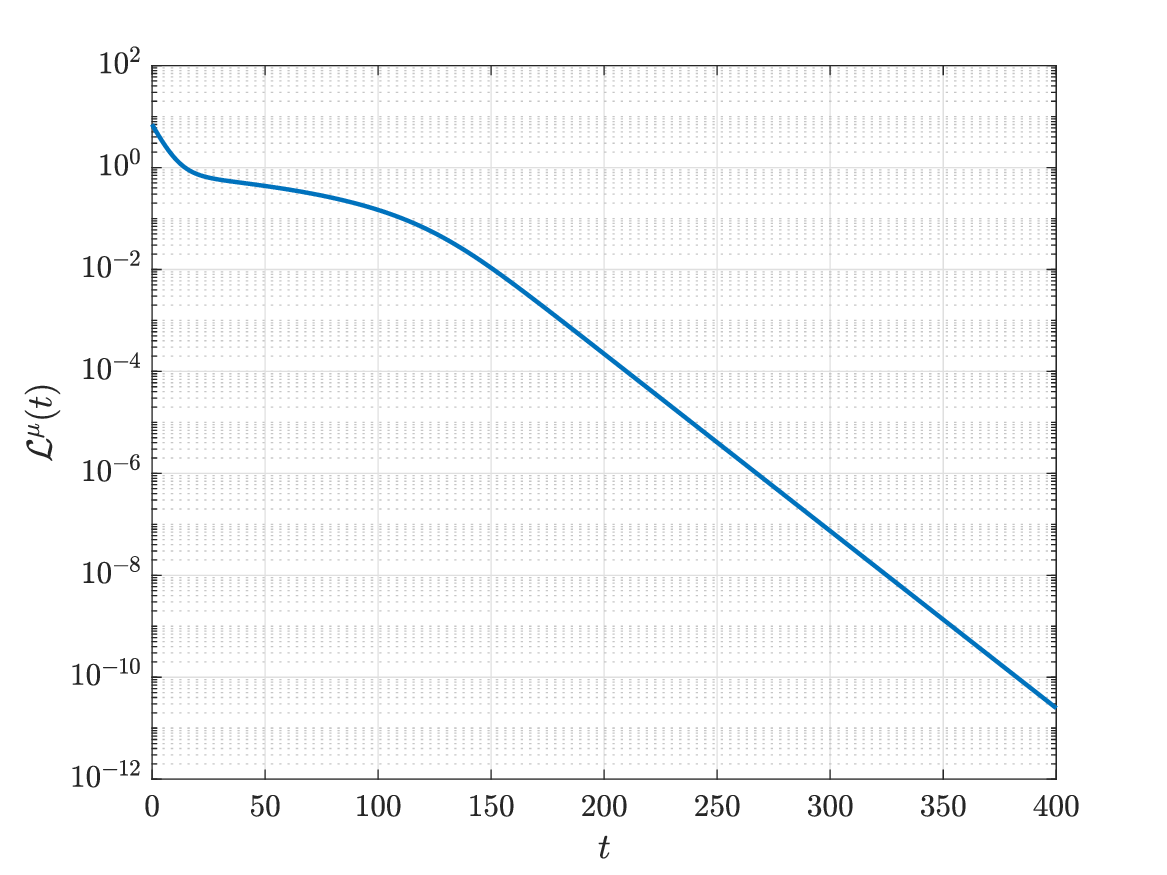}
	\caption{Effect of full control on the system. Left column: microscopic trajectories (top) and mean-field density (bottom). Right column: Lyapunov function decays.}
	\label{fig:fullcontrol}
\end{figure}

With full control applied ($b_i = 1$ for every agent $i$), as shown in Figure \ref{fig:fullcontrol}, agents successfully reach the target state. The mean-field density concentrates around $c=1$, and the Lyapunov function decays until approximately to $10^{-10}$, confirming our theoretical results.

\begin{figure}[H]
	\centering
	\includegraphics[width=0.45\linewidth]{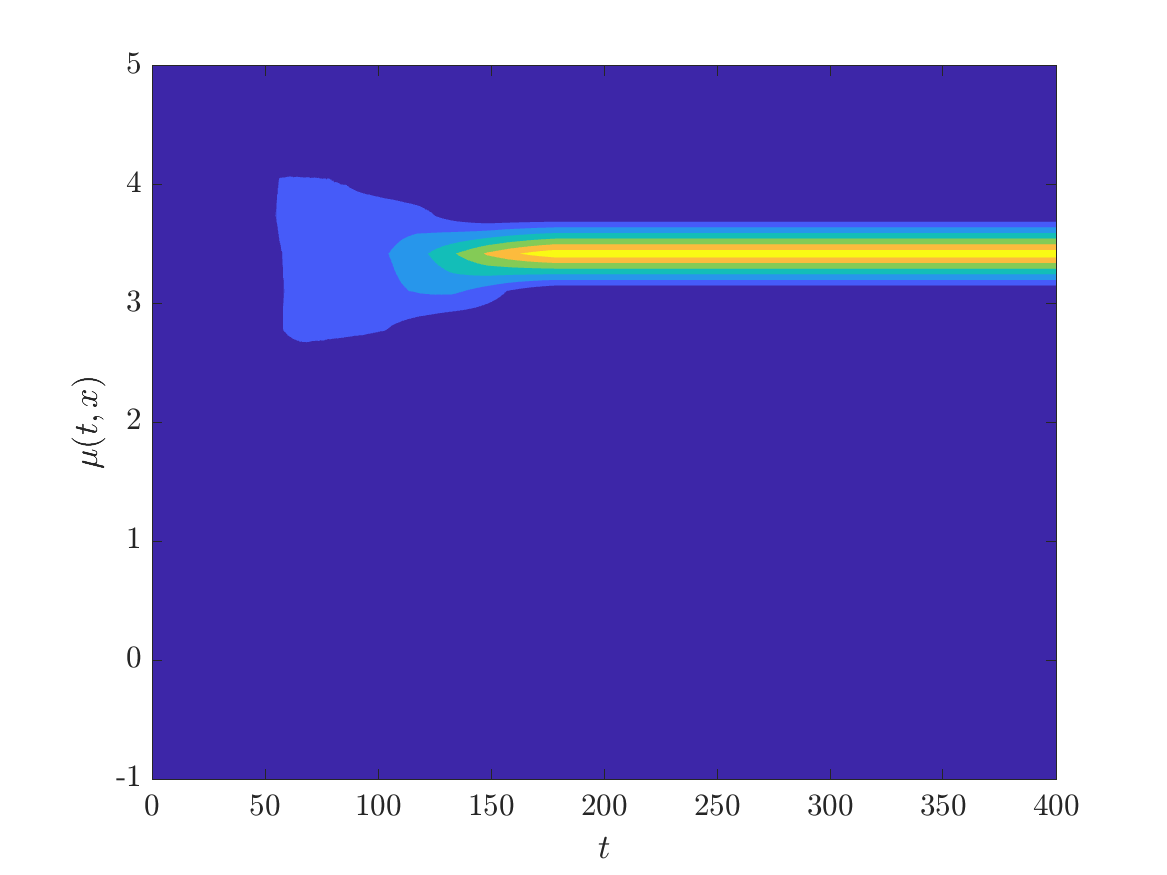}
	\includegraphics[width=0.45\linewidth]{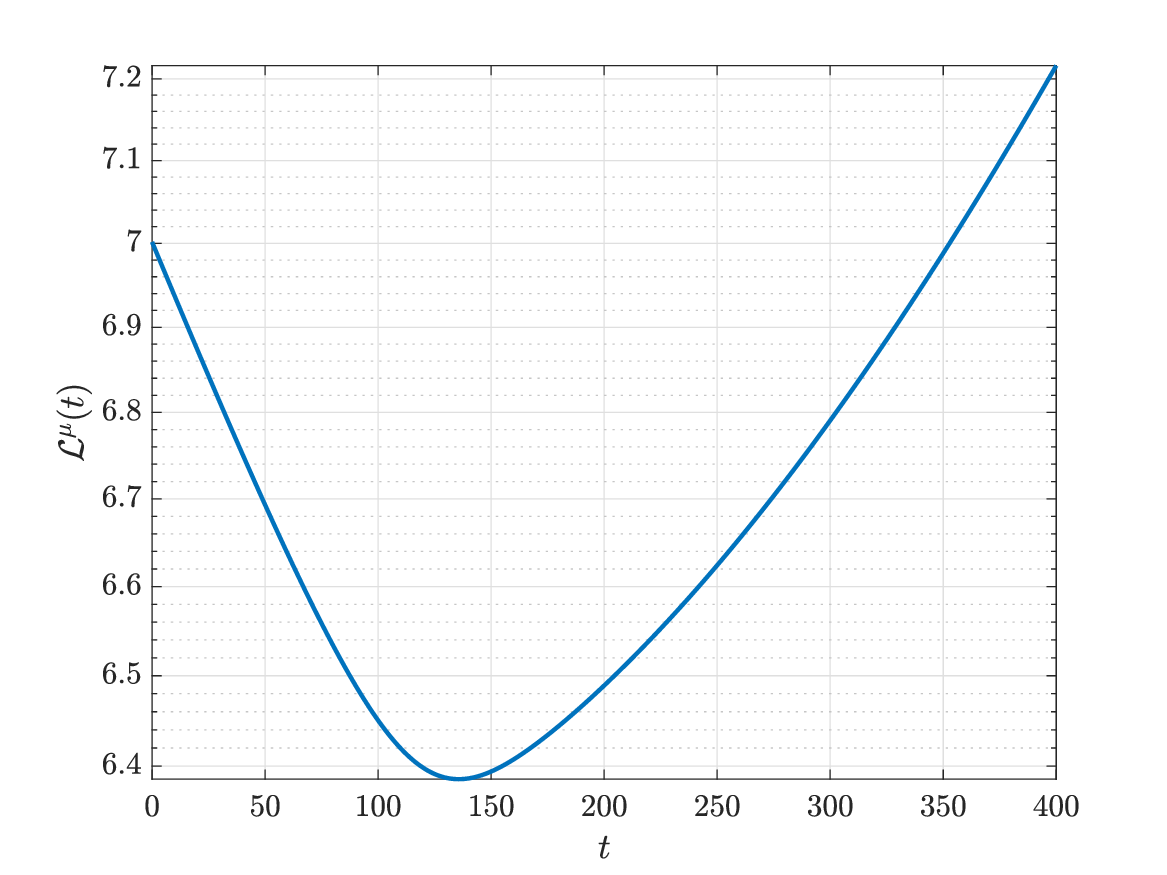}\\
	\includegraphics[width=0.45\linewidth]{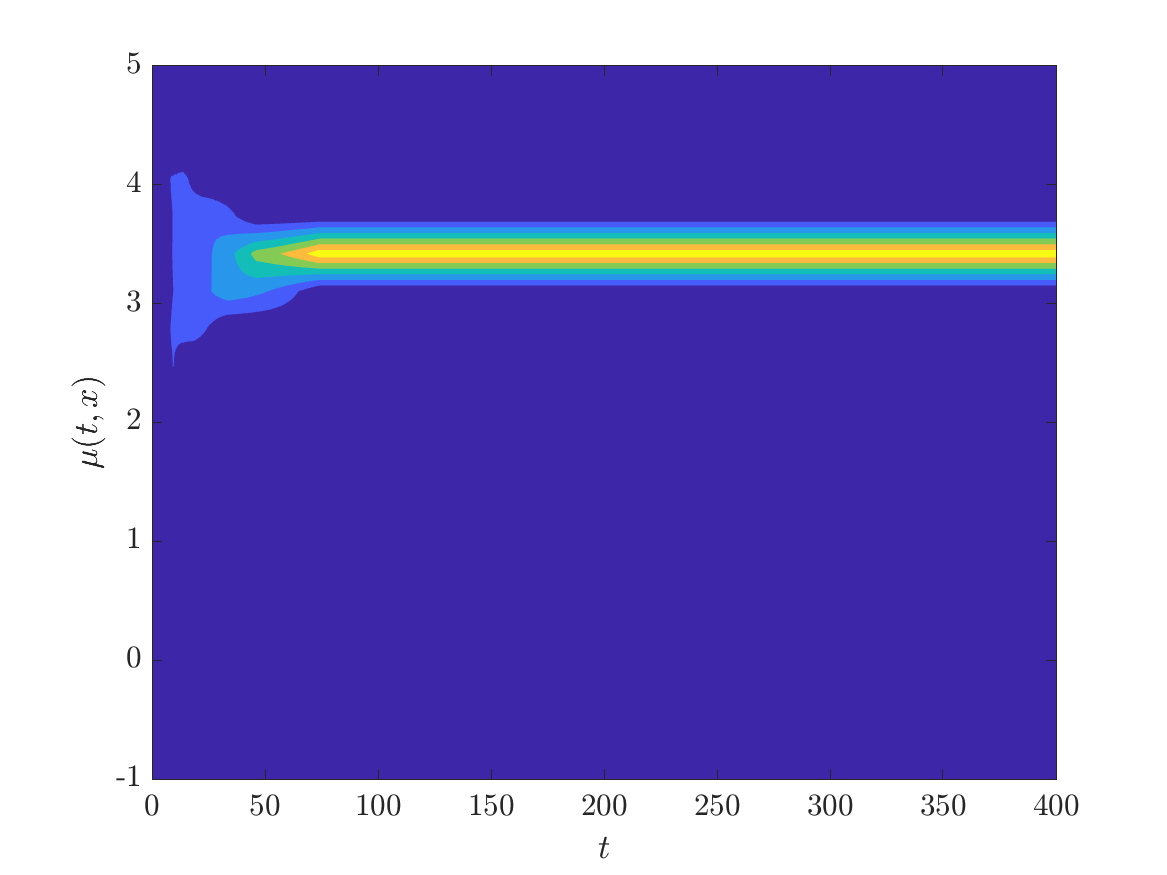}
	\includegraphics[width=0.45\linewidth]{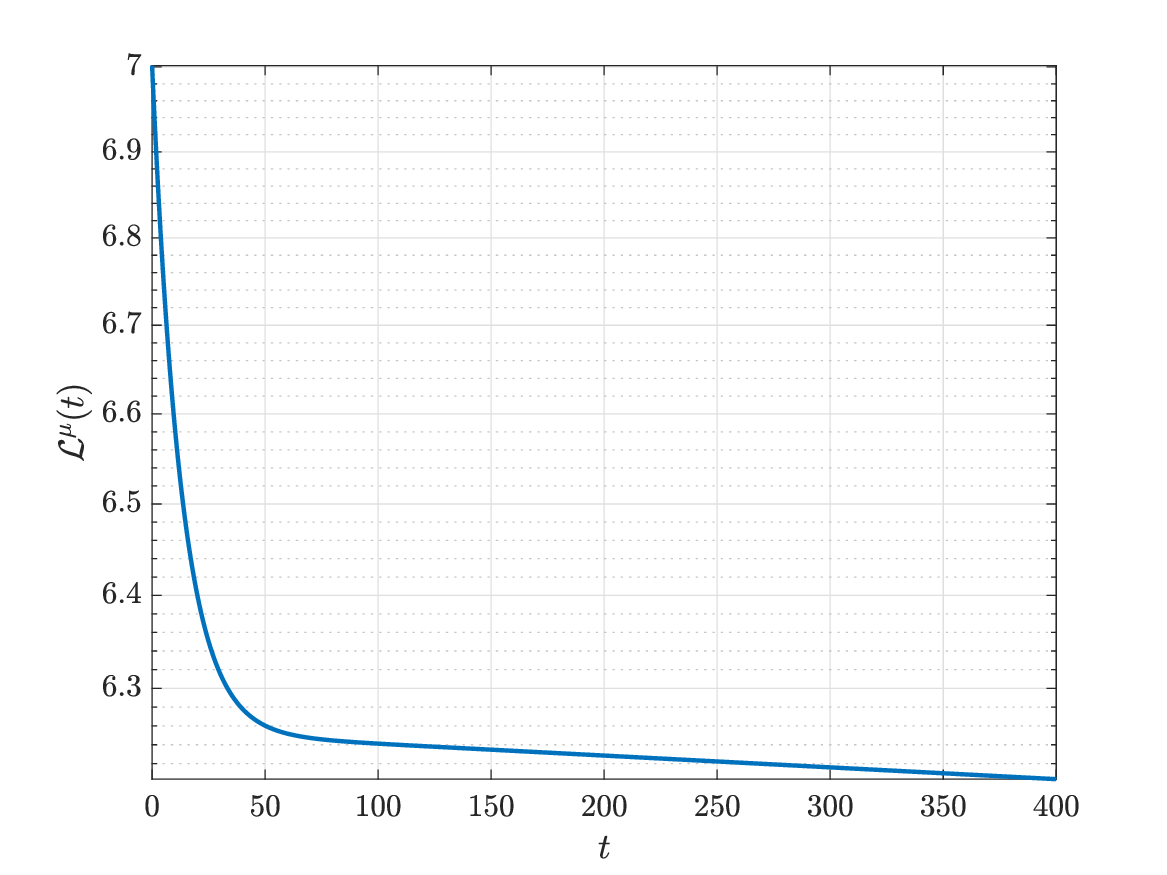}
	\caption{Sparse control dynamics. Left column: mean-field dynamics for the non-linear (top) and linear (bottom) case. Right column: Lyapunov function evolution over time.}
	\label{fig:sparsecontrol}
\end{figure}

Figure \ref{fig:sparsecontrol} presents the sparse control case ($b_i=1$ just for a random agent, all the others evolve without control action), focusing on the mean-field setting. Here, we display both the non-linear and linear dynamics, as we later aim to compare this scenario with the leader-follower model, for which we only have exponential decay results in the linear case. Notably, the agents do not reach the target, and the Lyapunov functional does not vanish, indicating incomplete stabilization.

The left column of the figure illustrates the evolution of the mean-field density under sparse control. In the non-linear case (top left), a single agent is directly controlled while the rest evolve through interaction dynamics. Since the overall population is large, the density plot does not reveal the deviation of the controlled agent, which is driven further below the target \(c\). The nonlinearity amplifies this effect, making it difficult to steer the entire system efficiently.  

The right column presents the evolution of the Lyapunov functional \(\mathcal{L}^\mu(t)\), which measures the squared distance from the target. In the non-linear case (top right), \(\mathcal{L}^\mu(t)\) initially decreases but starts increasing around \(t \approx 140\). This is due to the controlled agent moving too far below the target, while the uncontrolled agents remain clustered, leading to an overall divergence from the desired state.  

In the linear case (bottom row), the mean-field dynamics (bottom left) exhibit a more uniform evolution. However, as seen in the Lyapunov function plot (bottom right), \(\mathcal{L}^\mu(t)\) still does not reach zero, though it monotonically decreases. This suggests that sparse control is more effective in the linear setting but remains insufficient to fully drive the system to the target.

\begin{figure}[H]
	\centering
	\includegraphics[width=0.45\linewidth]{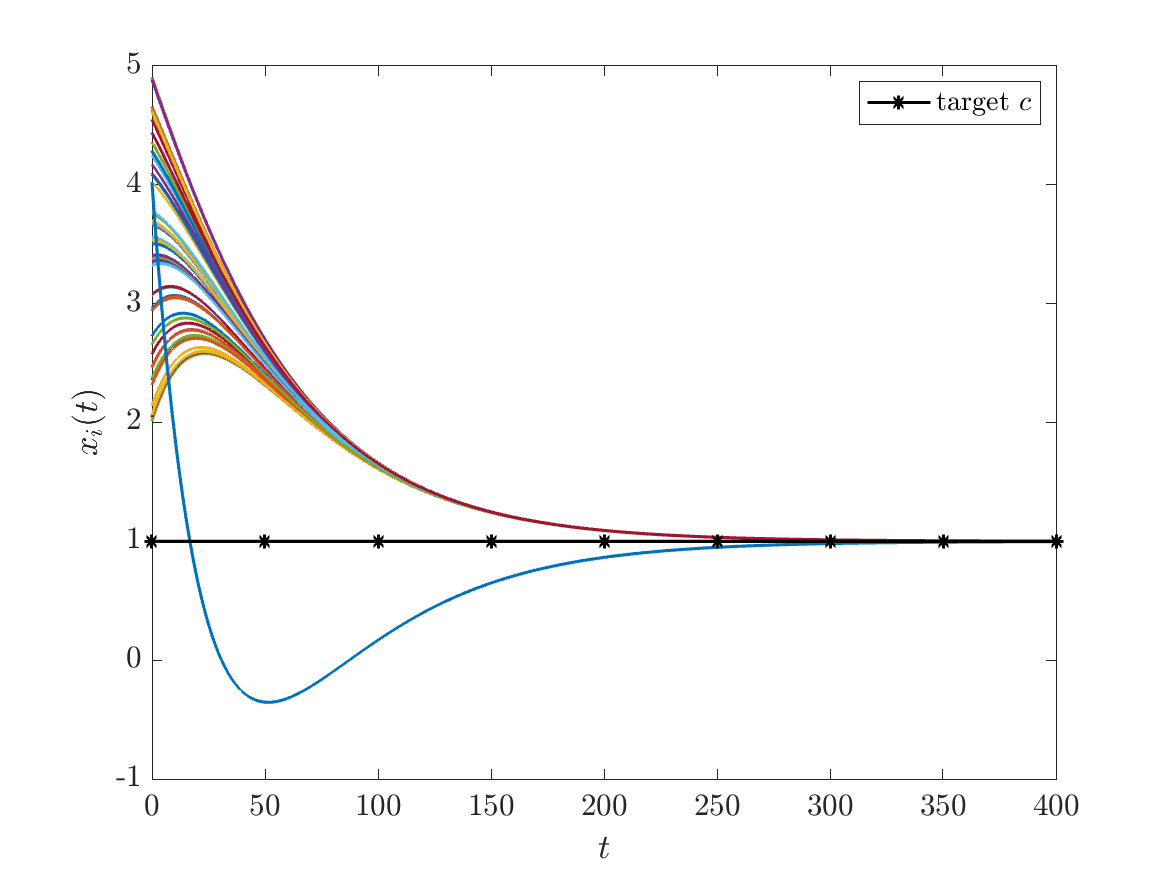}
	\includegraphics[width=0.45\linewidth]{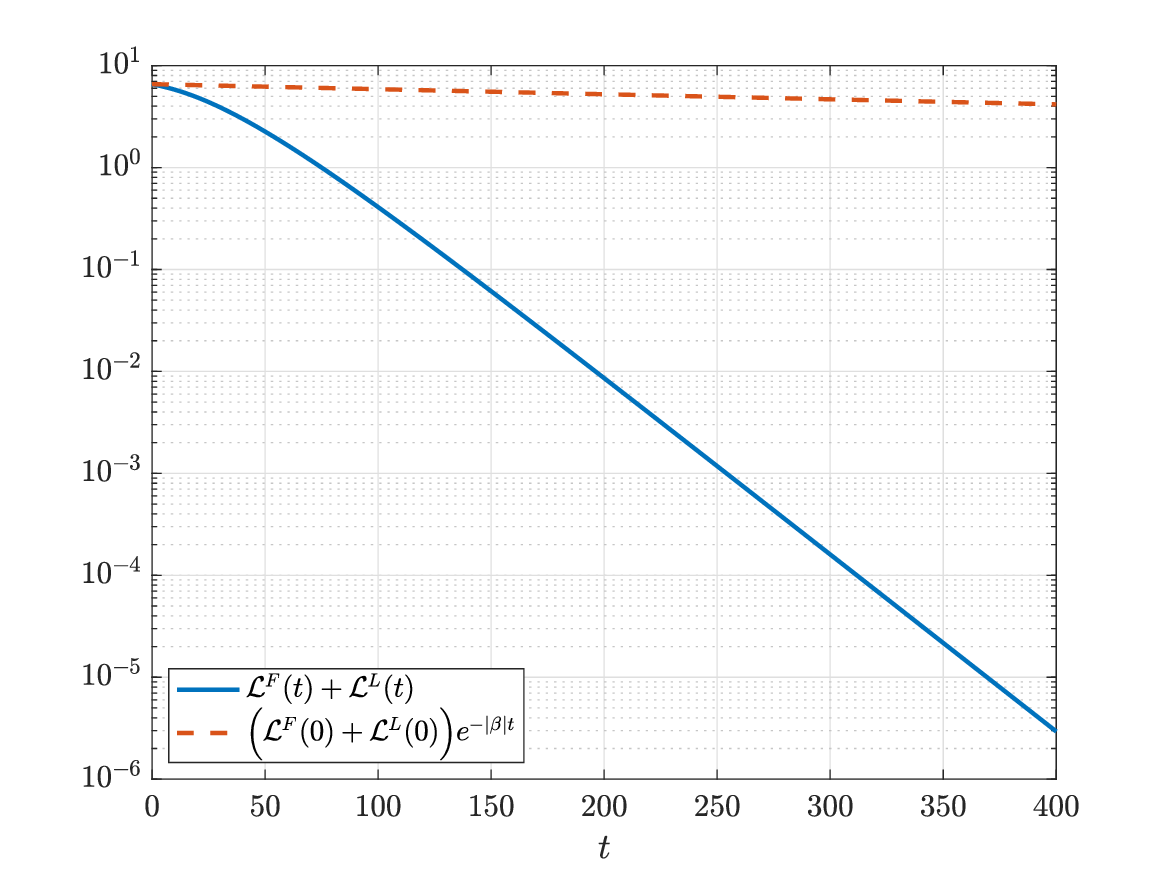}\\
	\includegraphics[width=0.45\linewidth]{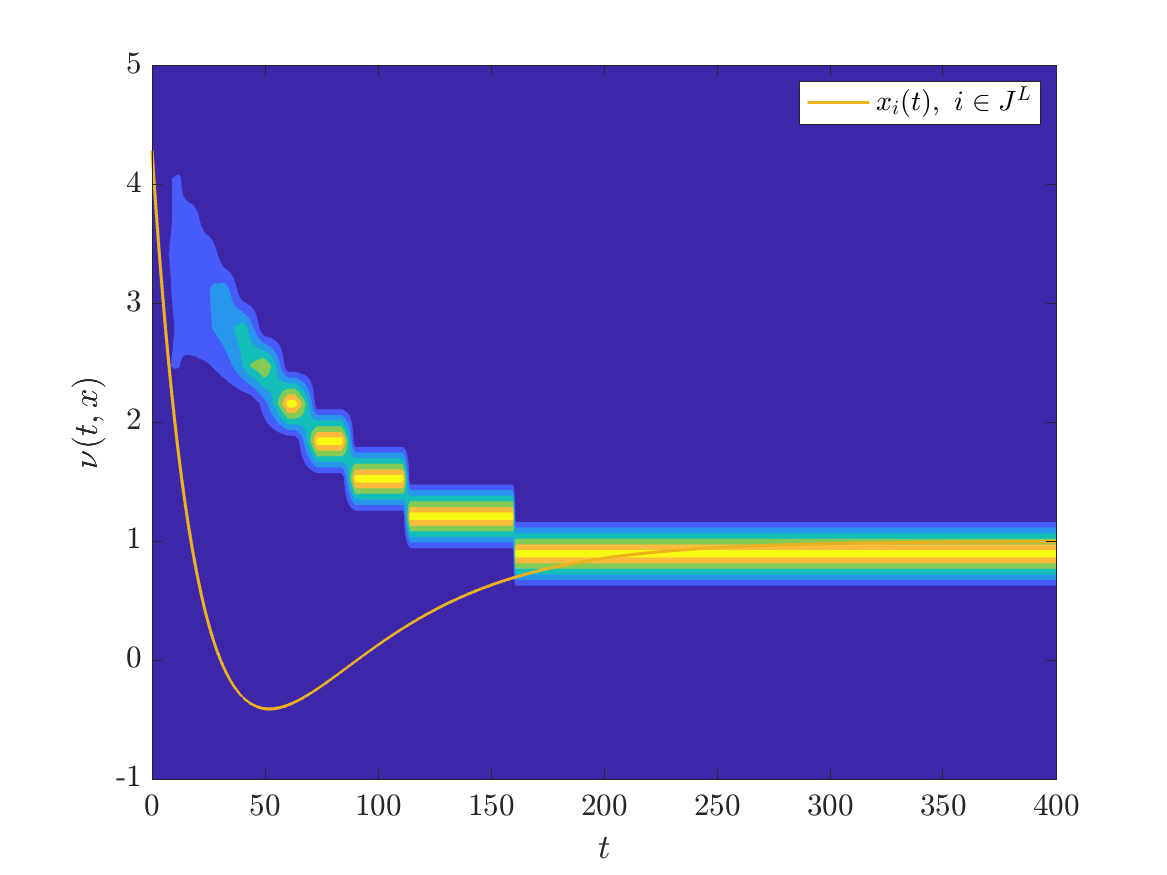}
	\includegraphics[width=0.45\linewidth]{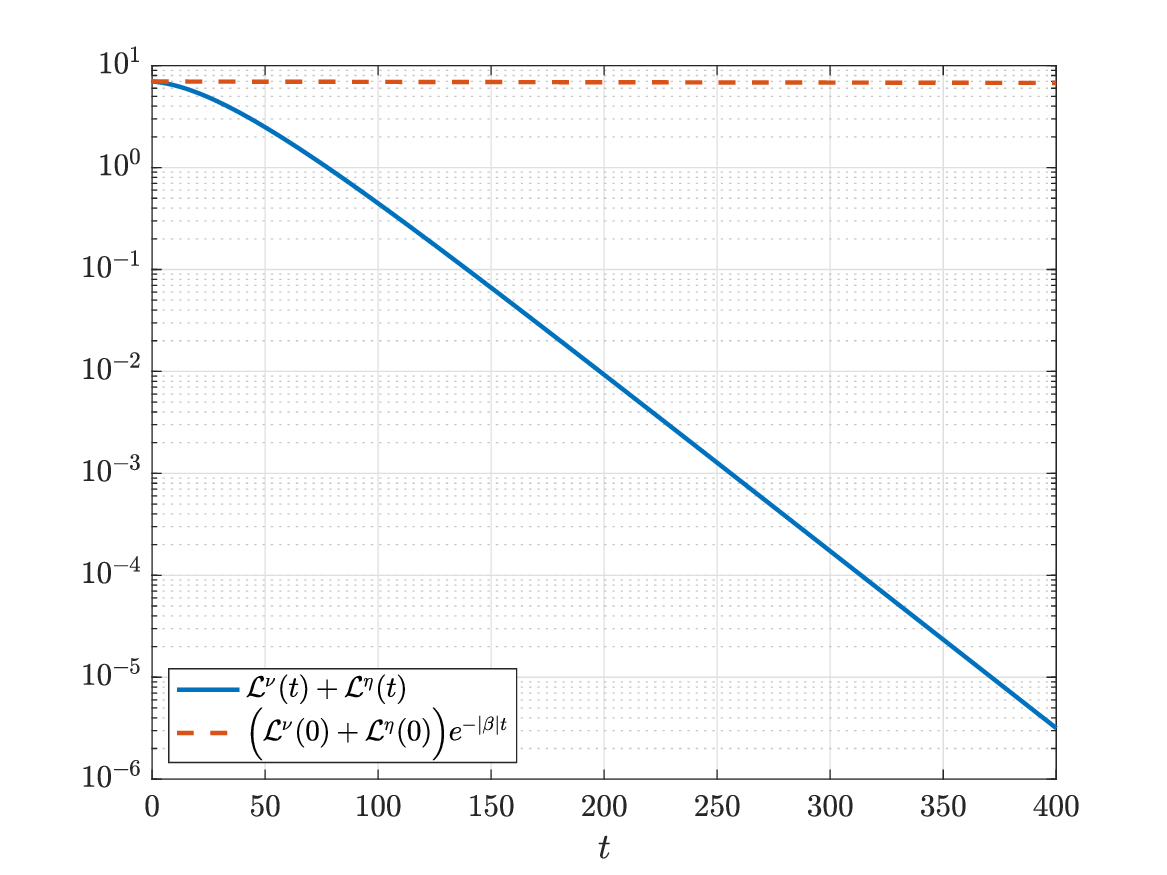}
	\caption{Leader-follower system. Left column: microscopic trajectories (top) and mean-field density (bottom). Right column: Lyapunov function decays.}
	\label{fig:leaderfollower}
\end{figure}

Figure \ref{fig:leaderfollower} illustrates the leader-follower setting, introduced through both the microscopic model \eqref{eq:leader_follower_micro} and the mean-field model \eqref{eq:leader_follower_mf}. Here, we consider the linear case where \( P(x_i, x_j) = \bar{p} \). We set $\rho^F = 0.9$, adjusting $\rho^L$ and the values of $\omega^F$ and $\omega^L$ accordingly. Even though the followers have much more weight, a single leader is still able to successfully guide all the followers to the target.
The Lyapunov function exhibits exponential decay, as predicted by the theory. To validate our theoretical results, we also display in a dashed red line the reference quantities 
\[
\Bigl(\mathcal{L}^F(0)+\mathcal{L}^L(0)\Bigr)e^{-|\beta| t}, \quad \Bigl(\mathcal{L}^\nu(0)+\mathcal{L}^\eta(0)\Bigr)e^{-|\beta| t},
\]  
with decay rate
\[
\beta = \left( 4 \bar{p} - 2 |k| \right) \frac{\sqrt{\omega^F \omega^L}}{2},
\]
demonstrating consistency with our propositions \ref{prop:micro_sparse_control} and \ref{prop:LF_MF}, given that the condition  
$|k| > 2 \bar{p}$
is satisfied in this case.

\begin{figure}[H]
	\centering
	\includegraphics[width=0.45\linewidth]{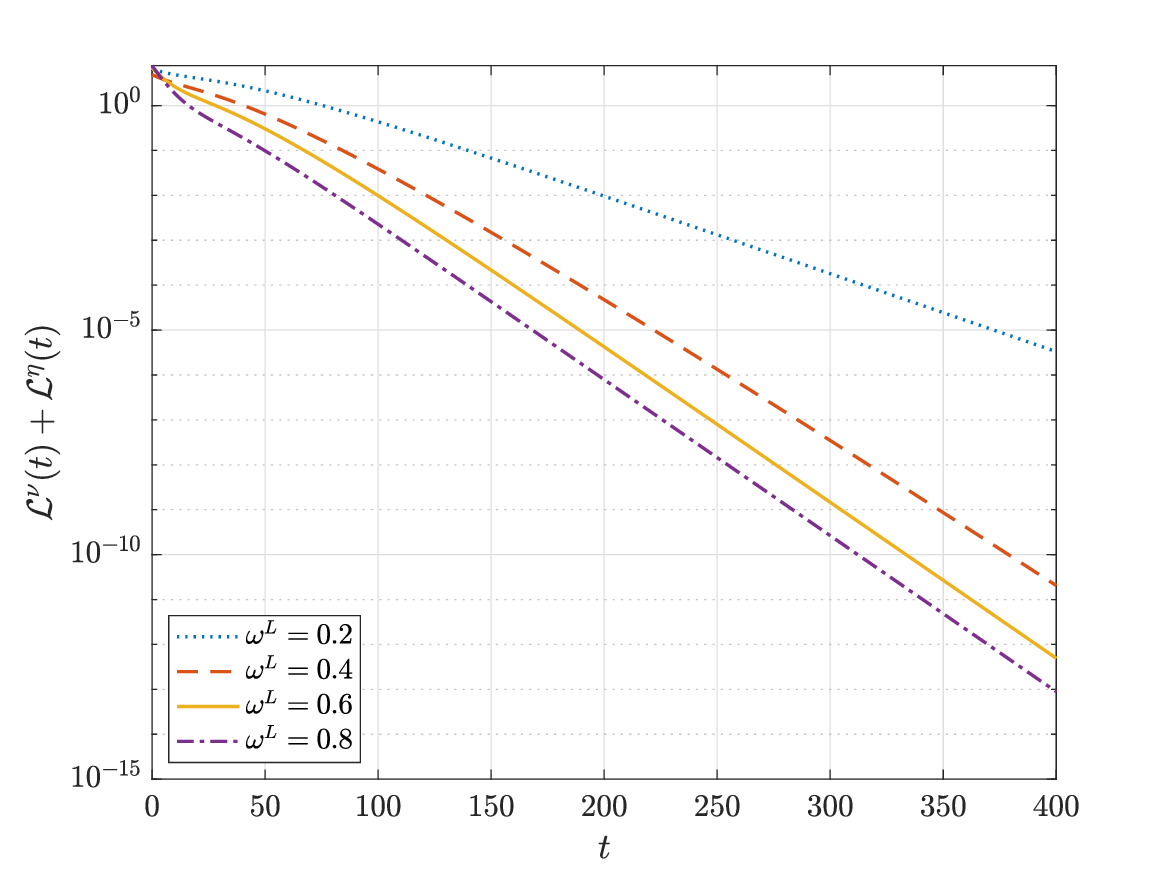}
	\caption{Lyapunov function decay for different values of $\omega^L$ under the leader-follower scenario.}
	\label{fig:lyapunov}
\end{figure}

Finally, Figure \ref{fig:lyapunov} shows a comparison of the decay of the Lyapunov functions in the leader-follower scenario, where the system consists of a single leader and a follower density. In this plot, we vary the weight assigned to the leader, denoted as $\omega^L$. The values of $\omega^L$ are set to 0.2, 0.4, 0.6, and 0.8, meaning the leader's influence increases as $\omega^L$ grows. As expected, when the leader is given more weight, the Lyapunov function decays more rapidly, indicating faster stabilization of the system.

\section{Conclusion}\label{sec:conclusion}
This work analyzed the stabilization of multi-agent systems through microscopic and mean-field models, with a focus on leader-follower dynamics and Lyapunov stability under full and sparse control. The microscopic framework established key stabilization properties, while the mean-field limit addressed large-scale systems with continuous agent densities.
Using Lyapunov-based techniques, we derived conditions ensuring decay for both control regimes, validated by eigenvalue analysis. Numerical tests confirmed these results, demonstrating the effectiveness and scalability of the proposed control strategies. These findings provide a robust foundation for further exploration of complex, large-scale multi-agent systems.

{\small \subsection*{Acknowledgments} 
The authors thank the Deutsche Forschungsgemeinschaft (DFG, German Research Foundation) for the financial support through 442047500/SFB1481 within the projects B04 (Sparsity fördernde Muster in kinetischen Hierarchien), B05 (Sparsifizierung zeitabhängiger Netzwerkflußprobleme mittels diskreter Optimierung) and B06 (Kinetische Theorie trifft algebraische Systemtheorie) and through SPP 2298 Theoretical Foundations of Deep Learning  within the Project(s) HE5386/23-1, Meanfield Theorie zur Analysis von Deep Learning Methoden (462234017). C. Segala ia a member of the Italian National Group of Scientific Calculus (Indam GNCS). C. Segala also thanks the Swiss National Science Foundation (SNSF) for the financial support through the grant number 215528, Large-scale kernel methods in financial economics.
}

\noindent \textbf{Conflict of Interest Statement}\\
All authors declare no conflicts of interest in this paper.

\bibliographystyle{siam}
\bibliography{refs2}

\end{document}